%% file: main.tex
\documentclass[11pt, twoside]{article}

\usepackage{geometry}
\input{command}

\begin{document}

\begin{center}

  {\bf{\LARGE{A Dual Accelerated Method for Online Stochastic Distributed Averaging: From Consensus to Decentralized Policy Evaluation\footnote{This work has been submitted to the IEEE for possible publication. Copyright may be transferred without notice, after which this version may no longer be accessible.}}}}

\vspace*{.2in}

{\large{
\begin{tabular}{ccc}
Sheng Zhang$^{\star}$, Ashwin Pananjady$^{\star, \dagger}$ and Justin Romberg$^{\dagger}$
\end{tabular}
}}
\vspace*{.2in}

\begin{tabular}{c}
Schools of Industrial and Systems Engineering$^\star$ and Electrical and Computer Engineering$^\dagger$ \\
Georgia Institute of Technology
\end{tabular}

\vspace*{.2in}

\today 

\vspace*{.2in}

\begin{abstract}
Motivated by decentralized sensing and policy evaluation problems, we consider a particular type of distributed stochastic optimization problem over a network, called the online stochastic distributed averaging problem. We design a dual-based method for this distributed consensus problem with Polyak--Ruppert averaging and analyze its behavior. We show that the proposed algorithm attains an accelerated deterministic error depending optimally on the condition number of the network, and also that it has an order-optimal stochastic error. This improves on the guarantees of state-of-the-art distributed stochastic optimization algorithms when specialized to this setting, and yields---among other things---corollaries for decentralized policy evaluation. Our proofs rely on explicitly studying the evolution of several relevant linear systems, and may be of independent interest. Numerical experiments are provided, which validate our theoretical results and demonstrate that our approach outperforms existing methods in finite-sample scenarios on several natural network topologies. 
\end{abstract}
\end{center}


\input{introduction}
\input{stochastic-dual-accelerated-method}
\input{main-results}
\input{numerical-results}
\input{proof-of-main-results}
\input{technical-lemmas}
\input{conclusion}

\small
\subsection*{Acknowledgments}
SZ was partially supported by the ARC/ACO Fellowship at Georgia Tech. AP was supported in part by the National Science Foundation through grants CCF-2107455 and DMS-2210734, and by an Adobe Data Science Faculty Research Award. JR was supported by ARL DCIST CRA W911NF-17-2-0181 and the AI4OPT NSF AI Institute.

\bibliography{refs.bib}

\normalsize
\appendix
\input{appendix-DSG}
\input{appendix-implementation}

\end{document}

%% file: introduction.tex
\section{Introduction}
\label{sec:introduction}
Consider an online, stochastic distributed averaging problem in which noisy data becomes available sequentially to agents residing on a network. Suppose the network is modeled as a simple, connected, and undirected graph $\mathcal{G} = \left(\mathcal{N}, \mathcal{E}\right)$ consisting of a set of nodes $\mathcal{N}=\{1,\ldots,N\}$ and a set of edges $\mathcal{E} \subseteq \mathcal{N} \times \mathcal{N}$. Here each edge $(i, j) \in \mathcal{E}$ should be thought of as an unordered pair of distinct nodes. The set of neighbors of node $i \in \mathcal{N}$ is denoted by $\mathcal{N}_i = \{j \in \mathcal{N} \;|\; (i, j) \in \mathcal{E}\}$. At every time step $t = 0, 1, 2, \ldots$, each node $i \in \mathcal{N}$ observes a local random vector $R_{i,t} \in \mathbb{R}^n$, with mean vector $\mathbb{E}[R_{i, t}] = \mu_i \in \mathbb{R}^n$ and covariance matrix $\text{Cov}\left[R_{i,t}\right] = \Sigma_{i} \in \mathbb{R}^{n \times n}$. Of particular interest will be the individual variances $\sigma^2_{i,1}, \ldots, \sigma^2_{i,n}$ of the $n$ components of $R_{i,t}$, which are also entries along the main diagonal of $\Sigma_i$. We assume that the local random vectors are generated independently across time and nodes. The goal is to iteratively estimate the average of the mean vectors $\bar{\mu}:=\frac{1}{N} \sum^{N}_{i=1} \mu_i$ at every node via a distributed algorithm in which the nodes can only communicate with their neighbors. We motivate this setting using the following two examples:
\paragraph{Distributed linear parameter estimation \citep{xiao2005scheme,kar2013distributed}}
Here we want to estimate a parameter vector $\beta^* \in \mathbb{R}^d$ using observations from a network $\mathcal{G} = \left(\mathcal{N}, \mathcal{E}\right)$ of $N$ sensors. At time $t \geq 0$, sensor $i \in \mathcal{N}$ makes an $m_i$-dimensional noisy measurement
\begin{align*}
Y_{i,t} = A_i \beta^* + \epsilon_{i,t},
\end{align*}
where $A_i$ is an $m_i\times d$ matrix known only to sensor $i$, and $\epsilon_{i,t} \in \mathbb{R}^{m_i}$ is a zero-mean noise vector that is independent across sensors and time, with covariance matrix $\Sigma_{\epsilon,i} \in \mathbb{R}^{m_i \times m_i}$. 

We work in the fully decentralized setting where each sensor can only exchange data with its neighbors; there is no central fusion center.  The goal is for each sensor to have an estimate of the unknown parameter $\beta^*$. To make the problem well-posed when $m_i<d$ for all $i \in \mathcal{N}$, we will assume that the matrix $\sum^N_{i=1} {A_i}^\top A_i$ is invertible (i.e. the ``distributed invertibility'' condition from \citep{kar2012distributed}). Since $\mu_i:= \mathbb{E}\left[A_i^\top Y_{i,t} \right] = A_i^\top A_i \beta^*$, we have
\begin{align*}
\beta^* 
&= \left(\frac{1}{N}\sum^N_{i=1} {A_i}^\top A_i\right)^{-1} \left(\frac{1}{N}\sum^N_{i=1} \mu_i \right).
\end{align*}
Thus, each sensor can form its own estimate of $\beta^*$ by estimating the global averages
\begin{align*}
\bar{A} := \frac{1}{N}\sum^N_{i=1} A_i^\top A_i \quad \text{and}\quad \bar{\mu}:=\frac{1}{N}\sum^N_{i=1} \mu_i,
\end{align*}
respectively. The problem of estimating $\bar{A}$ and $\bar{\mu}$ are covered by our setting, though we note that $\bar{A}$ is deterministic and can be estimated using standard distributed averaging algorithms. We will present finite-time bounds on how accurately $\beta^*$ can be approximated given each sensor's estimates of $\bar{A}$ and $\bar{\mu}$.

\paragraph{Decentralized multi-agent policy evaluation \citep{doan2019finite,zhang2021taming}}
A central problem in reinforcement learning (RL) is to estimate the value function of a given stationary policy in a Markov decision process, often referred to as the policy evaluation problem. Moreover, it is an important component of many policy optimization algorithms for RL. Because the policy is given and applied automatically to select actions, such a problem is naturally formulated as value function estimation in a Markov reward process (MRP). 

Here, $N$ agents operate in a common environment modeled by a finite MRP consisting of a set of states $\mathcal{S}=\{1,\cdots,n\}$, a state transition matrix $P \in [0,1]^{n \times n}$, rewards $r_i\in\mathbb{R}^n$ for agent $i \in \mathcal{N}$ being in each state, and a discount factor $\gamma \in [0, 1)$. We again work in the fully decentralized
setting where agents can only communicate with their neighbors on a network $\mathcal{G}=(\mathcal{N}, \mathcal{E})$. Their goal is to cooperatively estimate the value function $J^*: \mathcal{S} \rightarrow \mathbb{R}$ defined for all $s \in \mathcal{S}$ as 
\begin{align*}
J^*(s) := \mathbb{E} \left[\sum^\infty_{t=0} \gamma^t \bar{r}_{s_t} \Big| s_0 = s \right],
\end{align*}
where $s_{t+1} \sim P(s_t, \cdot)$ for all $t \geq 0$, and $\bar{r}_j$ is the $j$-th component of $\bar{r} := \frac{1}{N}\sum^N_{i=1} r_i$. It is known that $J^*$ solves the Bellman equation
\begin{align}
\label{eq:Bellman equation}
J = \bar{r} + \gamma P J,
\end{align}
meaning $J^* = \left(I - \gamma P \right)^{-1} \bar{r}$.

In the learning setting, the transition matrix $P$ and the local reward vectors $r_1, \ldots, r_N$ are unknown, and we instead assume access to a black box simulator. This observation model is often referred to as the \emph{generative model} \citep{kearns1999finite}: for each time step $t 
\geq 0$ and for each state $j \in \mathcal{S}$, each agent $i \in \mathcal{N}$ observes a common random next state $X_{t,j} \sim P(j, \cdot)$, and a local random reward $R_{i,t,j}$ with mean $r_{i,j}$ and variance $\sigma^{2}_{i,j}$. We assume that the local random rewards are generated independently across time and agents. A natural approach to this problem is to use the samples collected to construct estimates $\left(\widehat{P}, \widehat{r}\right)$ of the pair $\left(P, \bar{r}\right)$, and then plug these estimates into the Bellman equation \eqref{eq:Bellman equation}. The problem of estimating $\bar{r}$ in a decentralized manner is covered by our framework, and we will provide finite-sample bounds on how precisely $J^*$ can be approximated given each agent's estimates of the pair $\left(P, \bar{r}\right)$. 

\subsection{Related work}
There has been much recent interest in developing distributed algorithms for applications in robotics \citep{giordani2010distributed}, resource allocation \citep{xu2017distributed}, power system control \citep{peng2014distributed}, optimal transport \citep{krishnan2018distributed} and multi-agent reinforcement learning \citep{doan2019finite}. This is motivated mainly by the emergence of large-scale networks, where a huge amount of data is involved, and the generation and processing of information are not centralized. Notable among these are algorithms that can be used by a group of agents to reach a consensus in a distributed manner. The distributed consensus problem has been studied extensively in the computer science literature \citep{lynch1996distributed} and has found a number of applications including coordination of UAVs \citep{li2019distributed}, information processing in sensor networks \citep{zhang2018adaptive}, and distributed optimization \citep{tutunov2019distributed}.

The distributed averaging problem is a special case in which the goal is to compute the exact average of the initial values of the agents via a distributed algorithm. The most common distributed averaging algorithms are linear and iterative, which can be classified as deterministic or randomized. Several well-known deterministic distributed averaging algorithms were proposed and analyzed in the papers~\citep{tsitsiklis1984problems,tsitsiklis1986distributed,xiao2004fast,xiao2007distributed,olshevsky2009convergence}, where at each time step, every agent takes a weighted average of its own value with values received from its neighbors. For other deterministic algorithms, we refer the reader to \citet{olshevsky2009convergence} and the references therein. There are also two popular randomized algorithms, where at each time step, either two randomly selected nodes interchange information \citep{boyd2006randomized}, or a randomly selected node broadcasts its value to all its neighbors \citep{aysal2009broadcast}. For a discussion of other randomized algorithms, we refer the reader to \citet{dimakis2010gossip} and the references therein.

Many existing algorithms for distributed averaging require that agents are able to receive precise measurement values. However, constrained by limited sensing, agents might only be able to observe noisy measurements. Moreover, modern distributed systems involve a large amount of data available in sequential order. As each agent is subject to computation and storage constraints, it needs to process and distribute information received in an online fashion. Motivated by these considerations, in this paper, we study the natural online stochastic distributed averaging problem described above.

Our framework can be viewed as a special case of distributed stochastic optimization. The goal of distributed optimization is to minimize a global objective function given as a sum of local objective functions held by each agent, in a distributed manner. 
The distributed optimization problem has been studied for a long time and can be traced back to the seminal works \citep{tsitsiklis1984problems,tsitsiklis1986distributed} in the context of parallel and distributed computation. It has gained renewed prominence over the last decade due to its various applications in power systems \citep{molzahn2017survey}, communication networks \citep{li2018convergence}, machine learning \citep{nedic2020distributed}, and wireless sensor networks \citep{dougherty2016extremum}. Recent reviews can be found in the surveys \citep{nedic2018distributed,yang2019survey} and the books \citep{nedich2015convergence,giselsson2018large,bullo2020lectures}.

Distributed deterministic optimization is quite well understood with many centralized algorithms having their decentralized counterparts. For example, there exist distributed subgradient methods \citep{nedic2009distributed}, gradient methods \citep{shi2015extra}, and many variants of accelerated gradient methods \citep{scaman2017optimal,li2020variance,hendrikx2021optimal}, which achieve both communication and oracle complexity lower bounds. 

Optimal methods using a primal approach for smooth and strongly convex distributed stochastic optimization over networks were recently proposed and analyzed by \citet{fallah2019robust} and \citet{rogozin2021accelerated}. There are also methods using a dual approach \citep{gorbunov2019optimal,dvinskikh2019dual}, which are akin to the methods we develop and analyze in this paper (for the special class of quadratic functions). Sections \ref{sec:stochastic dual accelerated method} and \ref{sec:main results} provide detailed discussions of similarities and differences between our results and this body of work. In short, the dual approaches \citep{gorbunov2019optimal,dvinskikh2019dual} achieve optimal communication complexity in the general distributed optimization setting but fall short in terms of their oracle complexity.

\subsection{Contributions}
In this paper, we follow the dual approach of \citet{scaman2017optimal} and propose a stochastic dual accelerated method that uses constant step-size and Polyak–Ruppert averaging for the online stochastic distributed averaging problem. We establish non-asymptotic convergence guarantees with explicit dependence on the network connectivity parameter and noise in the observations. Our analysis builds on a discrete-time dynamical system representation of the algorithm and relies on explicitly studying the evolution of several relevant linear systems, which may be of independent interest. Our mean-squared error upper bounds provide tight guarantees on the bias and variance terms for the algorithm. We show that \begin{itemize}
    \item the bias term decays linearly at an accelerated rate with exponent $\mathcal{O}\left(- \frac{T}{\sqrt{\kappa(L)}}\right)$, where $\kappa(L)$ is a conditioning parameter quantifying the connectivity of the network and its precise definition is given in Section \ref{sec:stochastic dual accelerated method},
    \item the variance term achieves the $\mathcal{O}\left( \frac{\sum^n_{j=1}\max_{i \in \mathcal{N}} \sigma^2_{i,j}}{T}\right)$ rate up to a higher-order term in $T$, and
    \item the convergence rate of the algorithm is optimal.
\end{itemize}
Moreover, we show that our method outperforms, both in theory and in simulation, the state-of-the-art primal accelerated method called \texttt{D-MASG} \citep{fallah2019robust} in a relevant non-asymptotic regime where $T \asymp \sqrt{\kappa(L)}$. Furthermore, when assuming that $\sigma^2_{i,j} = \sigma^2_{i',j}$ for all $i\not = i' \in \mathcal{N}$ and $j \in \{1,\ldots,n\}$, and letting $\sigma^2 := \sum^n_{j=1} \sigma^2_{i,j}$, we show that our method has optimal per-node oracle complexity $\mathcal{O}\left(\frac{\sigma^2}{\varepsilon}\right)$ and optimal communication complexity $\mathcal{O}\left(\sqrt{\kappa(L)} \ln\left(\frac{1}{\varepsilon}\right)\right)$, where $\varepsilon > 0$ is the desired accuracy. In contrast, when specialized to our setting, state-of-the-art dual accelerated methods \citep{gorbunov2019optimal}, namely, \texttt{R-RRMA-AC-SA$^2$} and \texttt{SSTM\_sc}, have the same communication complexity $\mathcal{O}\left(\sqrt{\kappa(L)} \ln\left(\frac{1}{\varepsilon}\right)\right)$ as our algorithm, but can only be shown to have much larger per-node oracle complexity $\mathcal{O}\left(\kappa(L) \frac{\sigma^2}{\varepsilon^2}\right)$. Finally, numerical experiments are provided, which validate our theoretical results and demonstrate that our approach outperforms existing methods in finite-sample scenarios on several natural network architectures. 

%% file: stochastic-dual-accelerated-method.tex
\section{Stochastic Dual Accelerated Method}
\label{sec:stochastic dual accelerated method}
In this section, we first cast distributed averaging as distributed optimization with quadratic local objective functions, which is a well-known correspondence. Then, we follow \cite{scaman2017optimal,scaman2018optimal} and use the dual formulation of the distributed optimization problem to design an algorithm that can be executed for the online stochastic distributed averaging problem. 

First, note that the target vector $\bar{\mu}=\frac{1}{N} \sum^N_{i=1} \mu_i$ is the unique optimal solution of the optimization problem
\begin{align}
\label{eq: minimization problem single agent}
\min_{\theta \in \mathbb{R}^n}~ \frac{1}{2}\sum^N_{i=1} \Vert \theta - \mu_i\Vert^2_2.
\end{align}
 
A standard way for solving problem \eqref{eq: minimization problem single agent} in a decentralized setting is rewriting the problem as
\begin{align}
\label{eq:minimization problem rewrite}
\min_{\theta_1=\cdots=\theta_N}~ \frac{1}{2}\sum^N_{i=1} \Vert \theta_i - \mu_i\Vert^2_2.
\end{align}
The problem \eqref{eq:minimization problem rewrite} can be solved using the \emph{distributed stochastic gradient method} (\texttt{DSG}). This is a non-accelerated primal method that, at every iteration, performs a stochastic gradient descent update at each node with respect to its local objective function, and then performs weighted averaging with the decision variables from its neighbors. See Appendix \ref{sec:distributed stochastic gradient} for the formal description of \texttt{DSG} and its theoretical guarantees.

In this paper, we consider a dual approach for problem \eqref{eq:minimization problem rewrite}, which uses a decentralized communication scheme based on the gossip algorithm \citep{boyd2006randomized,duchi2011dual}. More specifically, during a communication step, each node $i$ broadcasts an $n$-dimensional vector to its neighbors and then computes a linear combination of the values received from its neighbors: node $i$ sends $\theta_i$ to its neighbors and receives $\sum_{j \in \mathcal{N}_i} L_{i,j} \theta_j$. One round of communication over the network can be represented as multiplying the current estimates with a gossip matrix $L= [L_{i,j}] \in \mathbb{R}^{N \times N}$. In order to encode communication constraints imposed by the network, we impose the following standard assumptions on $L$ \citep{scaman2017optimal,scaman2018optimal}:
\begin{itemize}
    \item[1.] $L$ is symmetric and positive semi-definite,
    \item[2.] $L$ is defined on the edges of the network: $L_{i,j} \not = 0$ only if $i=j$ or $(i,j) \in \mathcal{E}$,
    \item[3.] The kernel (i.e. nullspace) of $L$ is the set of all constant vectors.
\end{itemize} 
The third condition will ensure consensus among agents and also allow us to rewrite the consensus agreement constraint $\theta_1=\cdots=\theta_N$ in a convenient fashion. Note that a simple choice of the gossip matrix---which underlies our choice of notation---is the graph Laplacian matrix for $\mathcal{G}$, but other choices satisfying the above conditions are also valid. We will denote by $0 = \lambda_{N}(L) < \lambda_{N-1}(L) \leq \cdots \leq \lambda_{1}(L)$ the spectrum of $L$. Let $\kappa(L):=  \frac{\lambda_1(L)}{\lambda_{N-1}(L)}$ be the ratio between the largest and the second smallest eigenvalues of $L$. This quantity is the condition number of $L$ in the space orthogonal to the kernel of $L$, and is a parameter characterizing the connectivity of the network and how fast the information is spread over the network. Since $L$ is a real symmetric matrix, by the spectral theorem, it can be decomposed as $L = Q \Lambda Q^\top$, where $\Lambda := \text{diag}(\lambda_1(L), \ldots, \lambda_{N}(L))$ is a diagonal matrix whose entries are the eigenvalues of $L$ and $Q$ is an orthogonal matrix whose $i$-th column is the eigenvector of $L$ associated with $\lambda_i(L)$. Such a decomposition is not unique when the eigenvalues are not distinct; in this case, it suffices to choose any valid decomposition. 

We observe that the equality constraint $\theta_1=\cdots=\theta_N$ is equivalent to $\left(\sqrt{L} \otimes I_{n} \right) \Theta = 0$, where $\sqrt{L} := Q \sqrt{\Lambda} Q^\top$, $I_n$ is the $n \times n$ identity matrix, $\otimes$ is the the Kronecker product and $\Theta := \begin{bmatrix}
\theta_1^\top \cdots
\theta_N^\top
\end{bmatrix}^\top \in \mathbb{R}^{Nn}$. Here, we use $\sqrt{L}$ instead of $L$ because we will later square it via the change of variables. This observation leads to the following primal problem:
\begin{align}
\label{eq:primal problem distributed optimization}
\begin{aligned}
\min_{\Theta \in \mathbb{R}^{Nn}} \quad &
\frac{1}{2}\sum^N_{i=1} \Vert \theta_i - \mu_i\Vert^2_2\\
\text{s.t.} \quad & \left(\sqrt{L} \otimes I_{n} \right) \Theta = 0.\\
\end{aligned}
\end{align}
The Lagrangian function $\ell$ associated with problem \eqref{eq:primal problem distributed optimization} is given by
\begin{align*}
\ell(\Theta, \lambda) &=
\frac{1}{2}\sum^N_{i=1} \Vert \theta_i - \mu_i\Vert^2_2 - \lambda^\top \left[\left(\sqrt{L} \otimes I_{n} \right) \Theta\right] = \sum^N_{i=1} \left[\frac{1}{2}\Vert \theta_i - \mu_i\Vert^2_2 - {x_i}^\top \theta_i \right],\\
\end{align*}
where $\lambda \in \mathbb{R}^{Nn}$ is the Lagrange multiplier vector and $X:= \begin{bmatrix}
x_1^\top \cdots x_N^\top
\end{bmatrix}^\top = \left(\sqrt{L} \otimes I_{n} \right) \lambda$. Since strong duality holds, the convex program~\eqref{eq:primal problem distributed optimization} can be equivalently written in its dual form
\begin{align}
\label{dual problem: distributed average}
\max_{\lambda \in \mathbb{R}^{Nn}}\min_{\Theta \in \mathbb{R}^{Nn}} \ell(\Theta, \lambda)= -\min_{{\lambda \in \mathbb{R}^{Nn}}} \left\{\frac{1}{2} \lambda^\top \left(L \otimes I_{n} \right) \lambda + \left[\left(\sqrt{L} \otimes I_{n} \right) \mu \right]^\top \lambda \right \},
\end{align}
where $\mu := \begin{bmatrix}
\mu_1^\top \cdots \mu_N^\top 
\end{bmatrix}^\top$. To solve problem \eqref{dual problem: distributed average}, we may simply use gradient descent. Note that a gradient step with step-size $\eta >0 $ for problem \eqref{dual problem: distributed average} is $\lambda_{t+1}=\lambda_{t}-\eta \left(\sqrt{L} \otimes I_{n} \right)\left[ \left(\sqrt{L} \otimes I_{n} \right)\lambda_t + \mu \right]$,
and the change of variables $X_t= \left(\sqrt{L} \otimes I_{n} \right) \lambda_t$ yields the iteration
\begin{align}
\label{gradient step change of variable: distributed average}
X_{t+1}=X_{t}-\eta \left(L \otimes I_{n} \right) \left(X_t + \mu \right).
\end{align}

This update can be interpreted as gossiping the gradients of the local conjugate functions, since the convex conjugate of the local objective function $\frac{1}{2} \Vert \theta - \mu_i\Vert^2_2$ of node $i$ at $x \in \mathbb{R}^n$ is given by $\frac{1}{2}x^\top x + x^\top \mu_i$, and its gradient at $x_{i,t} \in \mathbb{R}^n$ is $x_{i,t} + \mu_i$, which is the $i$-th $n$-dimensional block of $X_t + \mu$. Since $\mu$ is unknown in the online stochastic distributed averaging problem, Eq.~\eqref{gradient step change of variable: distributed average} is not directly applicable. However, we have access to samples $\{R_{1,t}, \ldots, R_{N,t}\}$ at every time step $t \geq 0$. Thus, a natural way to obtain the stochastic version of iteration \eqref{gradient step change of variable: distributed average} is to replace $\mu$ with its unbiased estimator $\widehat{\mu}_t := \begin{bmatrix}
R_{1,t}^\top \cdots R_{N,t}^\top 
\end{bmatrix}^\top$.

\begin{algorithm}[ht]
\caption{Stochastic Dual Accelerated Algorithm (\texttt{SDA})}\label{alg:Stochastic Dual Accelerated Algorithm}
\begin{algorithmic}[1]
\State \textbf{Input:} number of iterations $T>0$, ``burn-in'' time $T_0 \in [0, T-1]$, gossip matrix $L\in\mathbb{R}^{N\times N}$, step-size $\eta > 0$ and momentum
parameter $\zeta \geq 0$.
\State \textbf{Initialization:} each agent $i \in \mathcal{N}$  initializes $x_{i,0} = y_{i,0} = \mathbf{0} \in \mathbb{R}^{n}$.
\For{$t=0,\cdots,T-1$}
\For{agent $i \in \mathcal{N}$}
\State observes a local random vector $R_{i,t}$ and executes the local update: \begin{equation*}
\label{eq:Stochastic Dual Accelerated Algorithm Stochastic Local Conjugate Gradient}
\begin{split}
\theta_{i,t} = x_{i,t} + R_{i,t}.
\end{split} 
\end{equation*}
\State exchanges $\theta_{i,t}$ with each agent $j \in \mathcal{N}_i$ and executes the local updates:
\begin{subequations}
\label{eq:Stochastic Dual Accelerated Algorithm NAG Step}
\begin{align*}
y_{i,t+1} &= x_{i,t} - \eta \sum_{j \in \mathcal{N}_i \cup \{i\}} L_{i,j} \theta_{j,t},\\
x_{i,t+1} &= y_{i,t+1} + \zeta \left(y_{i,t+1} - y_{i,t} \right).
\end{align*}
\end{subequations}
\EndFor
\EndFor
\State \textbf{Output:} $\widehat{\theta}_{i,T} := \frac{1}{T-T_0}\sum^{T-1}_{t=T_0}\theta_{i,t}$ for all $i \in \mathcal{N}$.
\end{algorithmic}
\end{algorithm}

While the above calculations provide transparent intuition on which to base algorithm design, our proposed algorithm is the \emph{stochastic dual accelerated method} (\texttt{SDA}) presented in Algorithm \ref{alg:Stochastic Dual Accelerated Algorithm}, which involves a more sophisticated (but still simple) iteration. In particular, it relies on Nesterov's accelerated gradient method \citep{nesterov2003introductory} with constant step-size, used in conjunction with Polyak--Ruppert averaging of the last $T - T_0$ iterates~\citep{polyak1992acceleration}. \texttt{SDA} is a stochastic variant of the single-step dual accelerated algorithm proposed and analyzed in \citet{scaman2017optimal}, which was developed for smooth and strongly convex distributed deterministic optimization. While both algorithms are similar in spirit, the analysis of \texttt{SDA} uses completely different techniques, since it applies to the stochastic setting for a special class of quadratic functions, as opposed to the deterministic setting for general smooth and strongly convex functions. The analysis in this paper builds on a dynamical system representation of the algorithm and relies on explicitly studying the evolution of several relevant linear systems, and may be of independent interest.

%% file: main-results.tex
\section{Main Results}
\label{sec:main results}
In this section, we begin by stating our theorem regarding the performance of \texttt{SDA}, and discussing some of the consequences of this result. Then we provide corollaries of the theorem for the two examples introduced in Section \ref{sec:introduction}. In order to state our theorem, we require the following definition: 

\begin{definition}
\label{definition:k^*}
$k^*$ is the smallest positive integer such that for all integer $k \geq k^*$:
\begin{align}
\label{eq:k^* definition}
\left(1+\frac{k}{\sqrt{\kappa(L)}+1}\right)\left(1 - \frac{1}{\sqrt{\kappa(L)}}\right)^{k} \leq \left(1 - \frac{1}{2\sqrt{\kappa(L)}}\right)^{k}. 
\end{align}
\end{definition}

\noindent Note that $k^*$ is well-defined and there exists an absolute constant $C \geq 1$ such that $k^* \leq C \cdot \sqrt{\kappa(L)}$ (see Lemma \ref{lemma:k^* upper bound} and its proof). Now, we are ready to present the finite-time performance bound of \texttt{SDA} in the following theorem.

\begin{theorem}
\label{theorem:upper bound of stochastic dual accelerated method}
Consider running \texttt{SDA} with the following parameters: 
\begin{align*}
T_0 = \frac{T}{2} \geq k^*, \quad \eta = \frac{1}{\lambda_1(L)} ~\text{ and }~  \zeta = \frac{\sqrt{\kappa(L)} - 1}{\sqrt{\kappa(L)} + 1},
\end{align*}
where $k^*$ is defined according to Definition \ref{definition:k^*}. Let $\left\{\widehat{\theta}_{i,T}\right\}_{i \in \mathcal{N}}$ be generated by \texttt{SDA}. Then we have
\begin{align}
\label{eq:finite-time upper bound of stochastic dual accelerated method}
\begin{split}
\mathbb{E}\left[\sum^N_{i=1} \norm{\widehat{\theta}_{i,T} - \bar{\mu}}^2_2\right] &\leq \frac{16\kappa(L)}{T^2} e^{-\frac{T}{2\sqrt{\kappa(L)}}} \sum^N_{i=1}\norm{\mu_i - \bar{\mu}}^2_2 \\
&\quad+ \frac{24 \left( k^* + \sqrt{\kappa(L)} \right) \sum^{N}_{i=1} \sum^n_{j=1} \sigma^2_{i,j}}{T^2} + \frac{2\sum^n_{j=1}\max_{i \in \mathcal{N}} \sigma^2_{i,j}}{T}.\\
\end{split}
\end{align}
\end{theorem}
\noindent See Section \ref{sec:proof of theorem 1} for a proof. It is worth making a few comments on this theorem. To simplify the discussion, we assume that $\sigma^2_{i,j} = \sigma^2_{i',j}$ for all $i \not = i' \in \mathcal{N}$ and $j \in \{1,\ldots,n\}$, and let $\sigma^2:=\sum^n_{j=1} \sigma^2_{i,j}$. Such an assumption is similar to the standard assumption in distributed stochastic optimization that the stochastic first-order oracle has finite variance bounded by $\sigma^2$, i.e., the variance of the stochastic (dual) gradient is bounded by $\sigma^2$.
\begin{remark}
Let us interpret the terms appearing in the bound \eqref{eq:finite-time upper bound of stochastic dual accelerated method}. Since $k^* \leq C \cdot \sqrt{\kappa(L)}$ for some absolute constant $C \geq 1$, the upper bound \eqref{eq:finite-time upper bound of stochastic dual accelerated method} simplifies to
\begin{align}
\label{eq:finite-time upper bound of stochastic dual accelerated method big-O}
\mathbb{E}\left[\sum^N_{i=1} \norm{\widehat{\theta}_{i,T} - \bar{\mu}}^2_2\right] &= \underbrace{\mathcal{O}\left(  e^{-\frac{T}{\sqrt{\kappa(L)}}}\sum^N_{i=1}\norm{\mu_i - \bar{\mu}}^2_2 \right)}_{\text{``bias''}} + \underbrace{\mathcal{O}\left(\frac{N \sqrt{\kappa(L)} \sigma^2}{T^2} + \frac{\sigma^2}{T}\right)}_{\text{``variance''}}.
\end{align}
This bound is presented in terms of two components: a bias term which is deterministic and independent of the noise level, and a variance term that measures the effect of noise on the algorithm. Note that our proposed method achieves an accelerated $\mathcal{O}\left(-\frac{T}{\sqrt{\kappa(L)}}\right)$ linear decay rate in the bias term---in the sense that it depends on $\sqrt{\kappa(L)}$ rather than $\kappa(L)$---as well as an $\mathcal{O}\left(\frac{\sigma^2}{T} \right)$ decay rate in the variance term up to the higher-order term in $T$. In fact, we can see from the variance term that the higher-order term in $T$, i.e. $\mathcal{O}\left(\frac{N \sqrt{\kappa(L)} \sigma^2}{T^2}\right)$, is dominated by $\frac{\sigma^2}{T}$ provided $T \gtrsim N \sqrt{\kappa(L)}$.
\end{remark}

\begin{remark}
We argue that the convergence rate of \texttt{SDA} is optimal. We first consider the noiseless setting where $\sigma^2 = 0$. It follows from Eq.~\eqref{eq:finite-time upper bound of stochastic dual accelerated method big-O} that our algorithm has a linear convergence rate $\mathcal{O}\left(-\frac{T}{\sqrt{\kappa(L)}}\right)$. The proof of Theorem 2 in \citet{scaman2017optimal} implies that there exist a gossip matrix $L$ and local functions $f_i$ in the special class of quadratic functions considered in Section \ref{sec:stochastic dual accelerated method} such that for any black-box distributed optimization algorithm using $L$, the convergence rate is at least $\Omega\left(e^{-\frac{T}{\sqrt{\kappa(L)}}}\right)$. Thus, the bias term achieves the optimal rate. Next, we consider the noisy setting where $\sigma^2 > 0$. Suppose the network is fully connected (or the graph is complete), then every node can be viewed as the center node that receives information from every other node and thus the distributed setting is reduced to the centralized setting for every node. Standard information-theoretic lower bounds on estimating Gaussian means then yield a lower bound $\Omega\left(\frac{\sigma^2}{T}\right)$ (see also classical results due to~\citet{nemirovskij1983problem}). Therefore, the rate of the variance term is optimal up to the higher-order term in $T$.
\end{remark}

\begin{remark}
The primal accelerated method named \texttt{D-MASG} \citep{fallah2019robust} is known to be asymptotically optimal for smooth and strongly convex distributed stochastic optimization problems~\citep{gorbunov2020recent}. If we apply Corollary 18 of \citet{fallah2019robust} to our setting, we obtain that the estimates $\{\widehat{\mu}_{i,T} \}_{i = 1}^N$ generated by this algorithm satisfy
\begin{align*}
\mathbb{E}\left[\sum^N_{i=1} \norm{\widehat{\mu}_{i,T} - \bar{\mu}}^2_2\right] = \mathcal{O}\left(e^{-\frac{T}{\sqrt{\kappa(L)}}} + \frac{N \kappa(L)^2 \sigma^2}{T^4} + \frac{\sigma^2}{T}\right).
\end{align*}
While the primal approach \texttt{D-MASG} and the dual approach \texttt{SDA} are rate-optimal for the online stochastic distributed averaging problem, their non-asymptotic behaviors can be significantly different. For example, in the non-asymptotic regime where $T \asymp \sqrt{\kappa(L)}$---which is the relevant regime given the accelerated rate of deterministic error---the upper bound on our algorithm scales as $\mathcal{O}\left(\frac{N\sigma^2}{T}\right)$, which is much better than the upper bound $\mathcal{O}\left(N\sigma^2\right)$ known to be achieved by \texttt{D-MASG}. 
\end{remark}

\begin{remark}
It is useful to compare with dual accelerated methods for smooth and strongly convex distributed stochastic optimization, \texttt{R-RRMA-AC-SA$^2$} and \texttt{SSTM\_sc} \citep{gorbunov2019optimal}. Applying Corollary 5.8 (for \texttt{R-RRMA-AC-SA$^2$}) and Corollary 5.14 (for \texttt{SSTM\_sc}) from \citet{gorbunov2019optimal} to our setting, the oracle complexity (the number of oracle calls per node) and communication complexity (the number of communication rounds) for both methods are $\mathcal{O}\left(\kappa(L) \frac{\sigma^2}{\varepsilon^2}\right)$ and $\mathcal{O}\left(\sqrt{\kappa(L)} \ln\left(\frac{1}{\varepsilon}\right)\right)$, where $\varepsilon > 0$ is the desired accuracy. Since these methods use batched stochastic dual gradients, let us make a small modification to \texttt{SDA} so as to facilitate a fair comparison. Specifically, suppose we change line 5 of \texttt{SDA} to: Each agent $i$ observes a batch local random vectors $\{R_{i, t, l}\}^{m_t}_{l=1}$ of size $m_t$, and executes local update
\begin{align}
\label{eq:Stochastic Dual Accelerated Algorithm Stochastic Local Conjugate Gradient batch}
\theta_{i,t} = x_{i,t} + \frac{1}{m_t} \sum^{m_t}_{l=1} R_{i,t, l}.
\end{align}
Under this modification, if we set $T=\mathcal{O}\left(\sqrt{\kappa(L)} \ln\left(\frac{1}{\varepsilon}\right)\right)$ and $m_t = \mathcal{O}\left(\frac{\sigma^2}{\varepsilon \sqrt{\kappa(L)}\ln\left(\frac{1}{\varepsilon}\right)}\right)$ for all $t=0, \ldots, T-1$, then Theorem \ref{theorem:upper bound of stochastic dual accelerated method} implies that $\mathbb{E}\left[\sum^N_{i=1} \norm{\widehat{\theta}^i_T - \bar{\mu}}^2_2\right] \leq \mathcal{O}\left(\varepsilon \right)$. Therefore, the batched version of our method has oracle complexity $\mathcal{O}\left(\frac{\sigma^2}{\varepsilon}\right)$ and communication complexity $\mathcal{O}\left(\sqrt{\kappa(L)} \ln\left(\frac{1}{\varepsilon}\right)\right)$.
While our method and the ones above have the same communication complexity, the oracle complexity of our method is much smaller since (i) it is independent of the condition number $\kappa(L)$ of the network; (ii) its dependence on $\varepsilon$ is $\frac{1}{\epsilon}$ instead of $\frac{1}{\epsilon^2}$. It is also worth noting that the batched version of our method achieves the oracle complexity lower bound $\Omega \left(\frac{\sigma^2}{\varepsilon}\right)$ and the communication complexity lower bound $\Omega \left(\sqrt{\kappa(L)} \ln\left(\frac{1}{\epsilon}\right)\right)$ simultaneously for our specific class of distributed stochastic optimization problems.  
\end{remark}

Next, we provide corollaries of Theorem \ref{theorem:upper bound of stochastic dual accelerated method} for the two examples introduced in Section~\ref{sec:introduction}.

\paragraph{Guarantees for distributed linear parameter estimation.}
We first cover how to generate an estimate of $\beta^* = \bar{A}^{-1} \bar{\mu}$ at each sensor in this setting by estimating the global averages $\bar{A}=\frac{1}{N}\sum^N_{i=1} A_i^\top A_i$ and $\bar{\mu}=\frac{1}{N}\sum^N_{i=1} \mu_i$, respectively. We first run \texttt{SDA} with local variables $R_{i,t} = {A_i}^\top A_i$ for all $t=0,\ldots,T' - 1$ and $i \in \mathcal{N}$ to obtain the estimates $\{\widehat{A}_{i,T'}\}_{i \in \mathcal{N}}$ of $\bar{A}$. Since $R_{i,t}$ is deterministic, iterate averaging is not necessary for the algorithm to converge. So we let \texttt{SDA} output the final iterates. Applying Theorem \ref{theorem:upper bound of stochastic dual accelerated method}, we obtain the linear convergence bound
\begin{align}
\label{eq:upper bound inverse of a matrix for corollary 1 (1)}
\sum^N_{i=1} \norm{\widehat{A}_{i,T'} - \bar{A}}^2_2 \leq \mathcal{O}\left(e^{-\frac{T'}{\sqrt{\kappa(L)}}}\sum^N_{i=1}\norm{A_i^\top A_i - \bar{A}}^2_2\right).
\end{align}
We assume that $T'$ is chosen large enough such that $\norm{\bar{A}^{-1}}_2 \norm{\widehat{A}_{i,T'} - \bar{A}}_2 \leq \frac{1}{2}$ for all $i \in \mathcal{N}$. Consequently, $\widehat{A}_{i,T'}$ is invertible for all $i \in \mathcal{N}$. Indeed, since $\bar{A}$ is assumed to be invertible, we have $\widehat{A}_{i,T'} = \bar{A} \left[I + \bar{A}^{-1}\left(\widehat{A}_{i,T'} - \bar{A}\right) \right]$, which implies that the invertibility of $\widehat{A}_{i,T'}$ is equivalent to the invertibility of $I + \bar{A}^{-1}\left(\widehat{A}_{i,T'} - \bar{A}\right)$. Since $\norm{\bar{A}^{-1} \left(\widehat{A}_{i,T'} - \bar{A}\right)}_2  \leq \norm{\bar{A}^{-1}}_2 \norm{\widehat{A}_{i,T'} - \bar{A}}_2 \leq \frac{1}{2}$, the matrix $I + \bar{A}^{-1}\left(\widehat{A}_{i,T'} - \bar{A}\right)$ has strictly positive eigenvalues and thus is invertible.

Next, we run \texttt{SDA} with local random variables $R_{i,t} = A_i^\top Y_{i,t}$ for all $t=0,\ldots,T-1$ and $i \in \mathcal{N}$ to obtain the estimates $\{\widehat{\mu}_{i,T}\}_{i \in \mathcal{N}}$ of $\bar{\mu}$. Noting that $R_{i,t}$ has zero mean and covariance matrix $A_i^\top \Sigma_{\varepsilon,i} A_i$, the individual variance $\sigma^2_{i,j}$ is the $j$-th element on the principal diagonal of $A_i^\top \Sigma_{\varepsilon,i} A_i$. Finally, each sensor $i$ computes its own estimate $\widehat{\beta}_{i, T, T'} := \left(\widehat{A}_{i,T'}\right)^{-1} \widehat{\mu}_{i,T}$ of the unknown parameter $\beta^*$. The next corollary states a finite-time error bound for this distributed sensing problem; see Section \ref{sec:proof of corollary 1} for a proof.

\begin{corollary}
\label{corollary:example I}
Consider running \texttt{SDA} with the following parameters: 
\begin{align*}
T_0 = \frac{T}{2} \geq k^*, \quad \eta = \frac{1}{\lambda_1(L)} ~\text{ and }~  \zeta = \frac{\sqrt{\kappa(L)} - 1}{\sqrt{\kappa(L)} + 1},
\end{align*}
where $k^*$ is defined according to Definition \ref{definition:k^*}. Let $\{\widehat{A}^i_T\}_{i \in \mathcal{N}}$ and $\{\widehat{\mu}^i_T\}_{i \in \mathcal{N}}$ be generated by \texttt{SDA} as described above, and $\widehat{\beta}_{i,T, T'} = \left(\widehat{A}_{i,T'}\right)^{-1} \widehat{\mu}_{i,T}$ denote the parameter estimated by sensor $i$. Then there exists a universal constant $C_1 > 0$ such that
\begin{align}
\label{eq:example I upper bound}
\begin{split}
&\mathbb{E}\left[\sum^N_{i=1} \norm{\widehat{\beta}_{i,T,T'}- \beta^*}^2_2\right] \leq C_1 \left\{\underbrace{\norm{\bar{\mu}}^2_2 \norm{\bar{A}^{-1}}^4_2 \left( e^{-\frac{T'}{\sqrt{\kappa(L)}}} \sum^N_{i=1}\norm{A_i^\top A_i - \bar{A}}^2_2\right)}_{\textbf{T}_1} \right.\\
&\left. + \underbrace{\norm{\bar{A}^{-1}}^2_2  \left( e^{-\frac{T}{\sqrt{\kappa(L)}}} \sum^N_{i=1}\norm{\mu_i - \bar{\mu}}^2_2 + \frac{\sqrt{\kappa(L)} \sum^{N}_{i=1} \sum^n_{j=1} \sigma^2_{i,j}}{T^2} + \frac{\sum^n_{j=1}\max_{i \in \mathcal{N}}\sigma^2_{i,j}}{T}\right)}_{\mathbf{\textbf{T}_2}}\right\},
\end{split}
\end{align}
\end{corollary}

Let us interpret the terms appearing in the upper bound \eqref{eq:example I upper bound}. The term $\textbf{T}_1$ bounds the mean squared error $\mathbb{E}\left[\sum^N_{i=1} \norm{\widehat{\beta}_{i,T,T'}- \beta^*}^2_2\right]$ when $\bar{\mu}$ is known, and measures how well we estimate $\bar{A}^{-1}$ at each sensor. More specifically, the factor $e^{-\frac{T'}{\sqrt{\kappa(L)}}} \sum^N_{i=1}\norm{A_i^\top A_i - \bar{A}}^2_2$ is due to the estimation error of $\bar{A}$ at each sensor (see Eq. \eqref{eq:upper bound inverse of a matrix for corollary 1 (1)}), which converges to zero exponentially fast with the optimal dependence on $\kappa(L)$, and the factor $\norm{\bar{A}^{-1}}^4_2$ is caused by inverting the matrices $\{\widehat{A}^i_T\}_{i \in \mathcal{N}}$ and arises because they are not precisely equal to $\bar{A}$ (see Eq. \eqref{eq:upper bound inverse of a matrix for corollary 1 (2)}). Furthermore, when $\bar{\mu}$ is known, we can factor $\bar{\mu}$ out of the difference $\widehat{\beta}_{i,T,T'}- \beta^*$ for each $i$, and thus the norm of the error term is proportional to the norm of $\bar{\mu}$.  

The term $\textbf{T}_2$ bounds the estimation error of $\beta^* = \bar{A}^{-1} \bar{\mu}$ when $\bar{A}^{-1}$ is known and measures how accurately each sensor estimates $\bar{\mu}$ with noisy measurements. The factor $\norm{\bar{A}^{-1}}^2_2$ accounts for the effect of matrix inversion as in standard linear regression, and the factor inside the parentheses is the error of estimating $\bar{\mu}$ at each sensor using the samples, a bound on which is directly implied by Theorem \ref{theorem:upper bound of stochastic dual accelerated method}. 

\paragraph{Guarantees for decentralized multi-agent policy evaluation.} In this setting, we first construct an unbiased estimator $\widehat{P}_T$ of the true transition matrix $P$ using the common state transition samples. For $t=0,\cdots,T-1$, each agent $i \in \mathcal{N}$ uses the set of sample state transitions $\{X_{t,j}|j\in \mathcal{S}\}$ to form a random binary matrix $Z_t \in \{0,1\}^{n \times n}$, in which row $j$ has a single non-zero entry corresponding to the index of the sample $X_{t,j}$. Thus, the location of the non-zero entry in row $j$ is drawn from the probability distribution $P(j, \cdot)$. Based on these observations, we define the common sample transition matrix $\widehat{P}_{T}:= \frac{1}{T} \sum^{T-1}_{t=0} Z_t$. Next, we run \texttt{SDA} with local random vectors $R_{i,t} := \left[R_{i,t,1}, \ldots, R_{i,t,n}\right]^\top$ for all $t=0,\ldots,T-1$ and $i \in \mathcal{N}$ to obtain estimates $\left\{\widehat{r}_{i,T}\right\}_{i \in \mathcal{N}}$ of $\bar{r}$. Finally, each agent $i$ ``plugs-in" the estimates $(\widehat{P}_T, \widehat{r}_{i,T})$ into the Bellman equation \eqref{eq:Bellman equation} to obtain the value function estimate $\widehat{J}_{i,T} := \left(I - \gamma \widehat{P}_T \right)^{-1} \widehat{r}_{i,T}$. The following corollary states a finite-time error bound for this decentralized policy evaluation algorithm; the proof is presented in Section \ref{sec:proof of corollary 2}.

\begin{corollary}
\label{corollary:example II}
Consider running \texttt{SDA} with the following parameters: 
\begin{align*}
T_0 = \frac{T}{2} \geq k^*, \quad \eta = \frac{1}{\lambda_1(L)} ~\text{ and }~  \zeta = \frac{\sqrt{\kappa(L)} - 1}{\sqrt{\kappa(L)} + 1},
\end{align*}
where $k^*$ is defined according to Definition \ref{definition:k^*}. Let $\{\widehat{r}_{i,T}\}_{i \in \mathcal{N}}$ be generated by \texttt{SDA}, $\widehat{P}_T$ denote the sample transition matrix defined above and $\widehat{J}_{i,T} = \left(I - \gamma \widehat{P}_T \right)^{-1} \widehat{r}_{i,T}$ represent the value function estimated by agent $i \in \mathcal{N}$. Then there exists a universal constant $C_2 > 0$ such that
\begin{align}
\label{eq:example II upper bound}
\begin{split}
\mathbb{E}\left[\max_{i \in \mathcal{N}}\norm{\widehat{J}_{i,T} - {J^*}}_{\infty} \right] \leq C_2 &\left\{e^{-\frac{T}{2\sqrt{\kappa(L)}}} \frac{\sqrt{\sum^N_{i=1}\norm{r_i - \bar{r}}^2_2}}{1 - \gamma} + \frac{\kappa(L)^{1/4} \sqrt{\sum^{N}_{i=1} \sum^n_{j=1} \sigma^2_{i,j}}}{(1 - \gamma)T} \right.\\
&\left. \quad + \frac{\sqrt{\sum^n_{j=1}\max_{i \in \mathcal{N}} \sigma^2_{i,j}}}{(1-\gamma)\sqrt{T}} + \frac{\sqrt{\log n}}{(1 - \gamma)^{3/2} \sqrt{T}} + \frac{\log n}{(1 - \gamma)^2 T} \right\}.
\end{split} 
\end{align}
\end{corollary}
It is instructive to compare our result with the result on distributed temporal-difference learning \citep{doan2019finite} for static networks\footnote{Note that strictly speaking, the two results are not comparable because we consider $\ell_\infty$-norm under a generative model whereas they consider a weighted $\ell_2$-metric under an i.i.d. observation model. Having said this, minimax rates of estimation are known to be comparable under both these observation models in the centralized setting for fast-mixing Markov chains~\citep{li2021sample}.}. First, we notice that their algorithm is not accelerated in terms of the network. Thus, our bias term converges much faster when the condition number of the network is large. Second, our bound scales as $\max\left\{\frac{\sqrt{\sum^n_{j=1}\max_{i \in \mathcal{N}} \sigma^2_{i,j}}}{(1-\gamma)\sqrt{T}}, \frac{\sqrt{\log n}}{(1-\gamma)^{3/2} \sqrt{T}}\right\}$ while theirs (Eq. (16) in their Theorem 1) scales as $\frac{1}{(1-\gamma)^{3/2} T^{1/4}}$ in the tabular setting. To see this, notice that the quantity $\beta_0$ defined in their Theorem 1 scales as $\frac{1}{(1-\gamma)^2}$. Consequently, our sample complexity in $1 - \gamma$ is tighter because to reach any given precision $\varepsilon > 0$, the samples per agent required for our algorithm is $\max \left\{\frac{\sum^n_{j=1}\max_{i \in \mathcal{N}} \sigma^2_{i,j}}{\varepsilon^2 \left(1 - \gamma \right)^2}, \frac{\log n}{\varepsilon^2 \left(1 - \gamma\right)^3}\right\}$ as opposed to $\frac{1}{\varepsilon^2 \left(1 - \gamma\right)^6}$ for their algorithm.

%% file: numerical-results.tex
\section{Numerical Results}
\label{sec:numerical results}
In this section, we conduct several experiments on synthetic examples to validate our theory and assess the performance of \texttt{SDA}. The implementation details are provided in Appendix \ref{sec:implementation details}. We simulate an online stochastic distributed averaging problem with $N$ nodes. At every time step $t \geq 0$, each node $i \in \mathcal{N}$ observes a local random variable $R_{i,t} \in \mathbb{R}$ (i.e. setting $n = 1$) drawn independently from the Gaussian distribution with mean $\mu_i \in \mathbb{R}$ and variance $\sigma^2=1$. We choose each $\mu_i$ by drawing independently from the uniform distribution over the interval $[0, b]$, where the constant $b \geq 0$ will be specified later in each set of experiments. We consider the following $5$ canonical network architectures:
\begin{enumerate}
    \item[(1)] Path graph: all $N$ nodes lie on a single straight line.
    \item[(2)] Cycle graph: a single cycle through all $N$ nodes.
    \item[(3)] Star graph: all other $N-1$ nodes connected to a single central node.
    \item[(4)] Square grid graph: all $N$ nodes form an $\sqrt{N} \times \sqrt{N}$ square grid.
    \item[(5)] A single realization of the Erdős–Rényi random graph: each edge is included in the graph with probability $p=\frac{2\ln{(N)}}{N}$, independently from every other edge. Here, the value of $p$ is chosen so that with high probability the realized graph is connected yet sparse.
\end{enumerate}

\begin{figure}[ht!]
\centering
\begin{subfigure}{0.27\textwidth}
    \includegraphics[width=\textwidth]{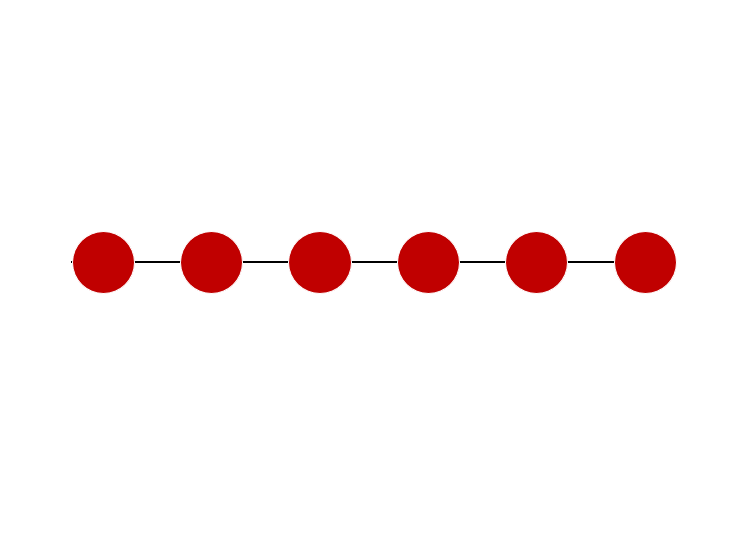}
    \caption{Path graph.}
    \label{fig:path graph}
\end{subfigure}
\hfill
\begin{subfigure}{0.3\textwidth}
    \includegraphics[width=\textwidth]{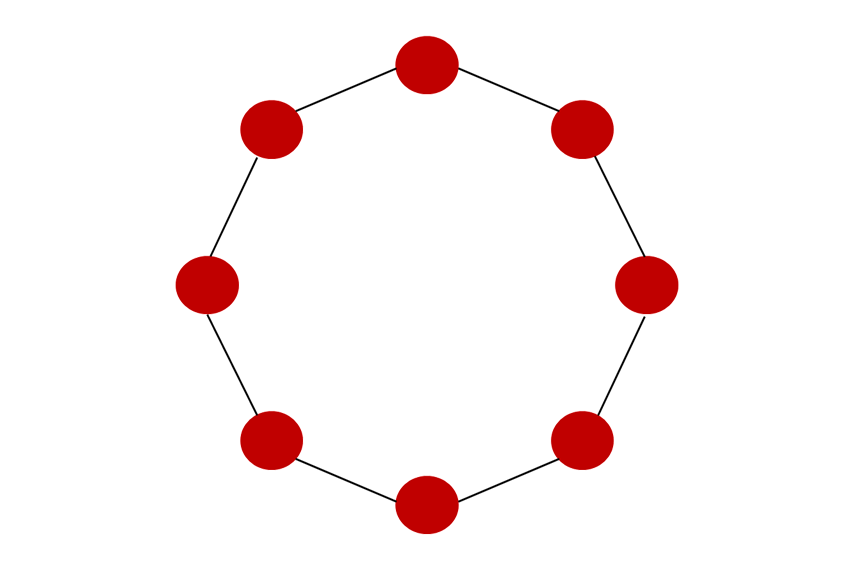}
    \caption{Cycle graph.}
    \label{fig:cycle graph}
\end{subfigure}
\hfill
\begin{subfigure}{0.3\textwidth}
    \includegraphics[width=\textwidth]{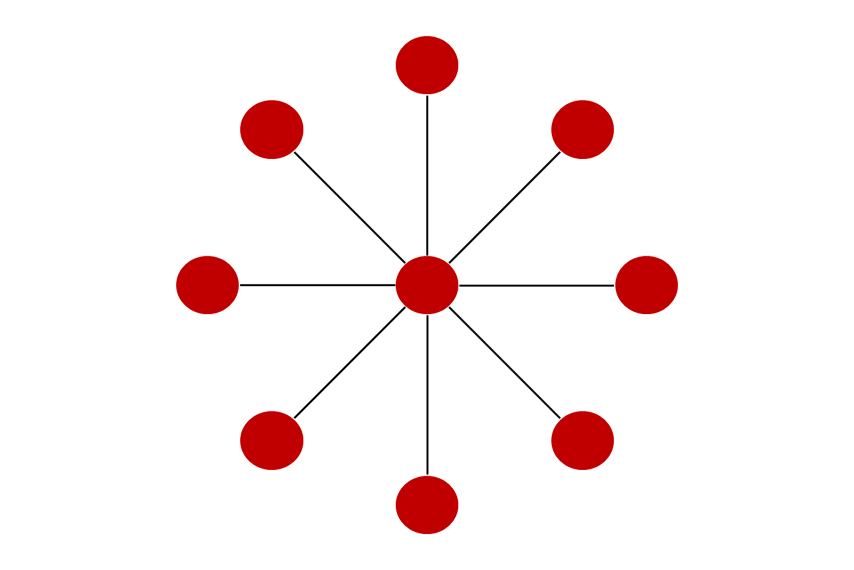}
    \caption{Star graph.}
    \label{fig:star graph}
\end{subfigure}
\begin{subfigure}{0.22\textwidth}
    \includegraphics[width=\textwidth]{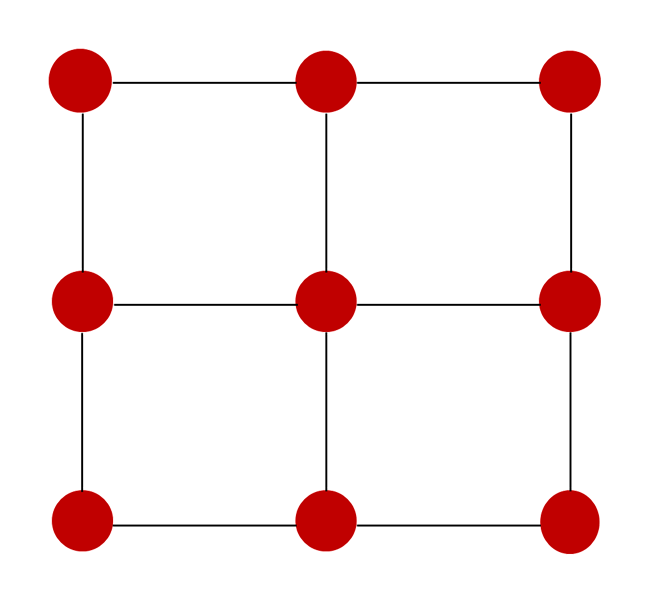}
    \caption{Square grid graph.}
    \label{fig:square grid graph}
\end{subfigure}
\hspace{1.5cm}
\begin{subfigure}{0.3\textwidth}
    \includegraphics[width=\textwidth]{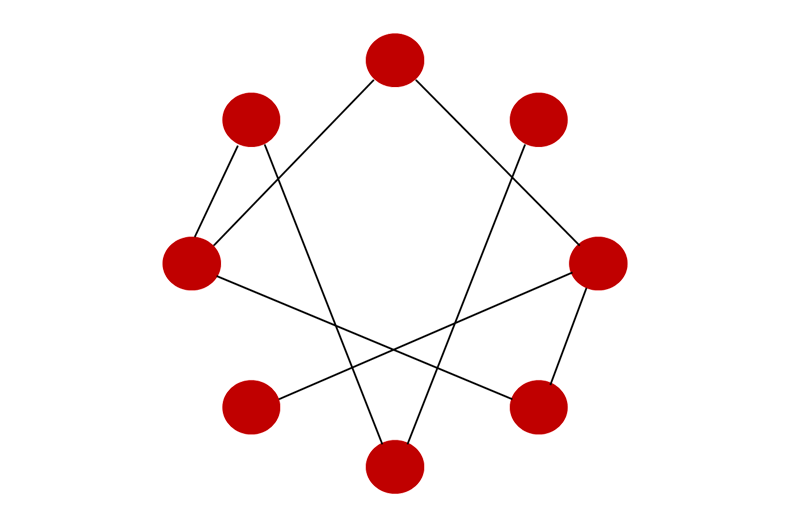}
    \caption{Erdős–Rényi random graph.}
    \label{fig:random graph}
\end{subfigure}
\vspace{0.3cm}
\caption{Illustration of the $5$ network architectures.}
\label{fig:5 networks}
\end{figure}

These graphs are illustrated in Figure \ref{fig:5 networks}. Note that, with the same number of nodes, the network connectivity from the path graph to the Erdős–Rényi random graph is increasing. In other words, when $N$ is fixed, we should expect $\kappa(L)$ to be decreasing from the path graph to the Erdős–Rényi random graph. For instance, in Figure \ref{fig:comparing with primal methods}--\ref{fig:sample complexity smaller epsilon SDA}, when $N=100$, the values of $\kappa(L)$ under the five network configurations are approximately $4052$, $1013$, $100$, $76$ and $11$ (with $p=\frac{2\ln{(N)}}{N}=0.0921$), respectively.

\subsection{Convergence behavior: \texttt{SDA} vs \texttt{D-MASG} vs \texttt{DSG}}
\label{sec:numerical results convergence behavior}
\begin{figure}[ht!]
\centering
\begin{subfigure}{0.46\textwidth}
    \includegraphics[width=\textwidth]{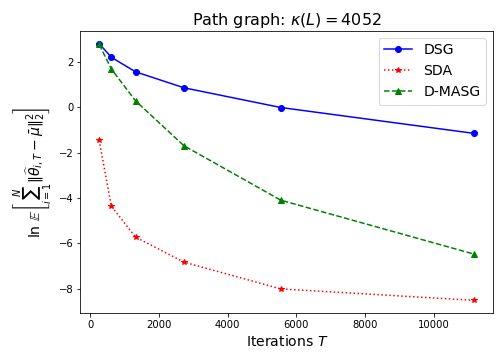}
    \label{fig:mse path graph}
\end{subfigure}
\hfill
\begin{subfigure}{0.46\textwidth}
    \includegraphics[width=\textwidth]{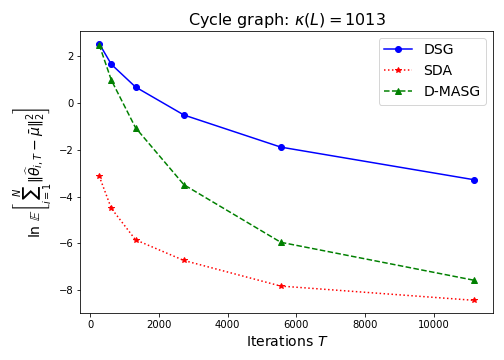}
    \label{fig:mse cycle graph}
\end{subfigure}

\begin{subfigure}{0.46\textwidth}
    \includegraphics[width=\textwidth]{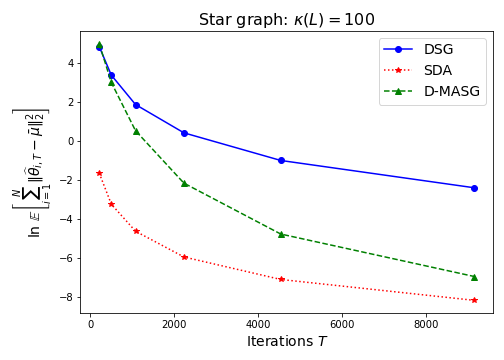}
    \label{fig:mse star graph}
\end{subfigure}
\hfill
\begin{subfigure}{0.46\textwidth}
    \includegraphics[width=\textwidth]{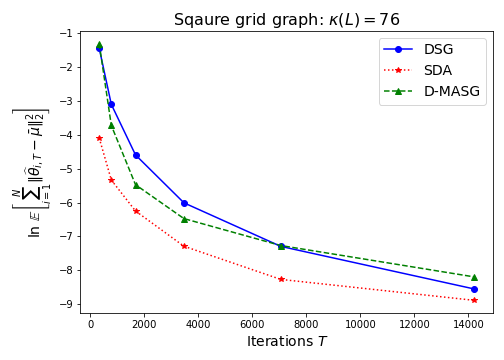}
    \label{fig:mse square grid graph}
\end{subfigure}

\begin{subfigure}{0.46\textwidth}
    \includegraphics[width=\textwidth]{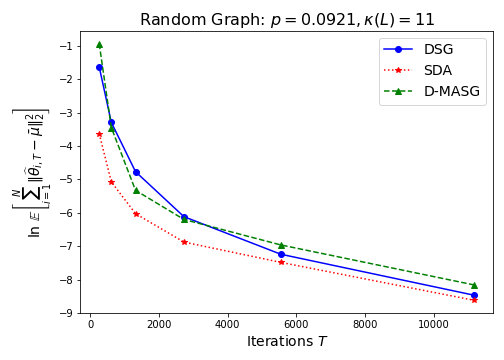}
    \label{fig:mse random graph}
\end{subfigure}
\caption{Convergence behavior of \texttt{SDA}, \texttt{DSG} and \texttt{D-MASG} on different graphs with $N=100$ nodes.}
\label{fig:comparing with primal methods}
\end{figure}

 In our first set of experiments, we compare \texttt{SDA} with two primal methods: \texttt{DSG} and \texttt{D-MASG}, where \texttt{DSG} is a non-accelerated method (see Appendix \ref{sec:distributed stochastic gradient}) while \texttt{D-MASG} \citep{fallah2019robust} is an accelerated method (see comment (3) of Section \ref{sec:main results}). Readers can see \texttt{DSG} as a baseline for this set of experiments. To show the convergence behavior of the three methods, we plot $\ln \mathbb{E}\left[\sum^N_{i=1} \norm{\widehat{\theta}_{i,T} - \bar{\mu}}^2_2\right]$ as a function of the number of iterations $T$ in Figure \ref{fig:comparing with primal methods}, where $\widehat{\theta}_{i,T}$ is the estimate of $\bar{\mu}$ by node $i$ after $T$ iterations of an algorithm. We can observe that (i) \texttt{SDA} outperforms the two primal algorithms on each graph and has significantly better performance when the network connectivity parameter $\kappa(L)$ is large; (ii) \texttt{SDA} and \texttt{D-MASG} converge much faster during their early iterations than \texttt{DSG} especially when $\kappa(L)$ is large, which demonstrates the acceleration of \texttt{SDA} and \texttt{D-MASG}. 

\subsection{Sample complexity: \texttt{SDA} vs \texttt{SSTM\_sc}}
\label{sec:numerical results sample complexity}
\begin{figure}[ht!]
\centering
\begin{subfigure}{0.46\textwidth}
    \includegraphics[width=\textwidth]{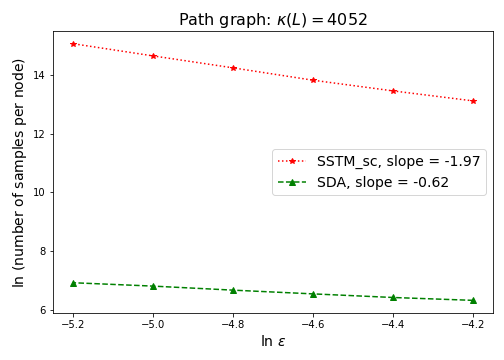}
    \label{fig:sample path graph}
\end{subfigure}
\hfill
\begin{subfigure}{0.46\textwidth}
    \includegraphics[width=\textwidth]{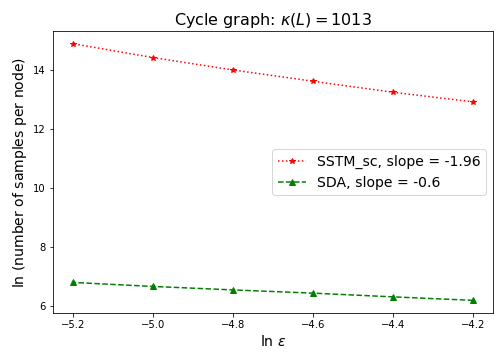}
    \label{fig:sample cycle graph}
\end{subfigure}

\begin{subfigure}{0.46\textwidth}
    \includegraphics[width=\textwidth]{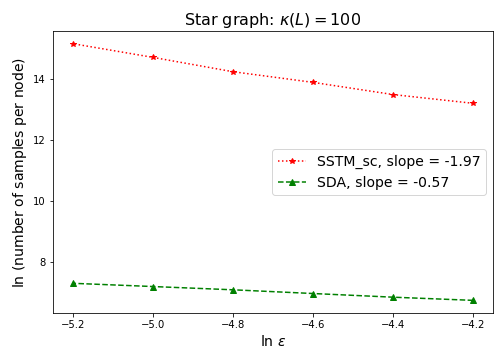}
    \label{fig:sample star graph}
\end{subfigure}
\hfill
\begin{subfigure}{0.46\textwidth}
    \includegraphics[width=\textwidth]{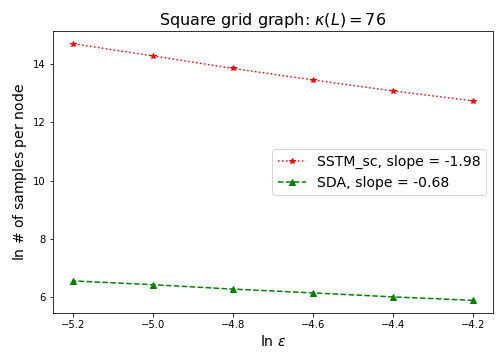}
    \label{fig:sample square grid graph}
\end{subfigure}

\begin{subfigure}{0.46\textwidth}
    \includegraphics[width=\textwidth]{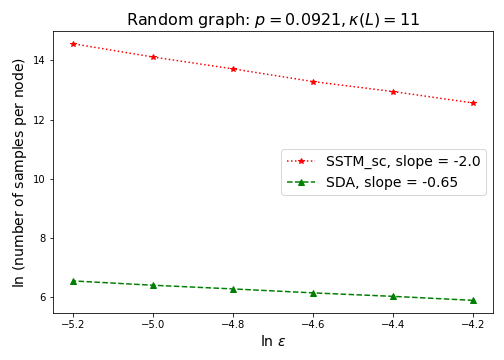}
    \label{fig:sample random graph}
\end{subfigure}
\caption{Sample complexity of \texttt{SDA} and \texttt{SSTM\_sc} on different graphs with $N=100$ nodes.}
\label{fig:sample complexity large epsilon}
\end{figure}

\begin{figure}[ht!]
\centering
\begin{subfigure}{0.46\textwidth}
    \includegraphics[width=\textwidth]{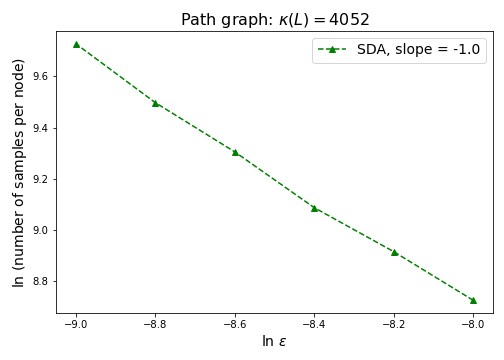}
    \label{fig:sample path graph SDA}
\end{subfigure}
\hfill
\begin{subfigure}{0.46\textwidth}
    \includegraphics[width=\textwidth]{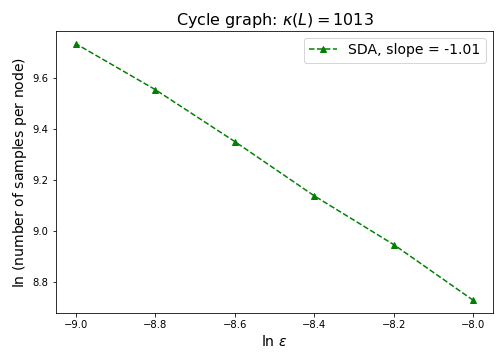}
    \label{fig:sample cycle graph SDA}
\end{subfigure}

\begin{subfigure}{0.46\textwidth}
    \includegraphics[width=\textwidth]{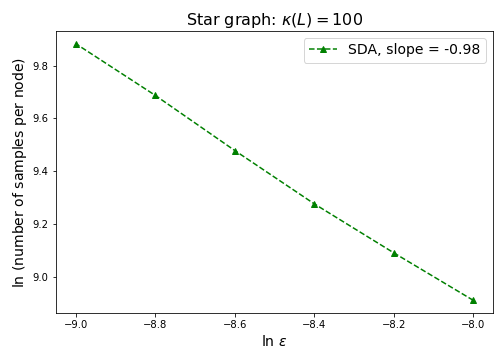}
    \label{fig:sample star graph SDA}
\end{subfigure}
\hfill
\begin{subfigure}{0.46\textwidth}
    \includegraphics[width=\textwidth]{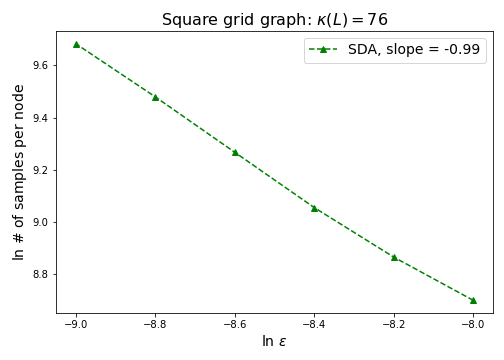}
    \label{fig:sample square grid graph SDA}
\end{subfigure}

\begin{subfigure}{0.46\textwidth}
    \includegraphics[width=\textwidth]{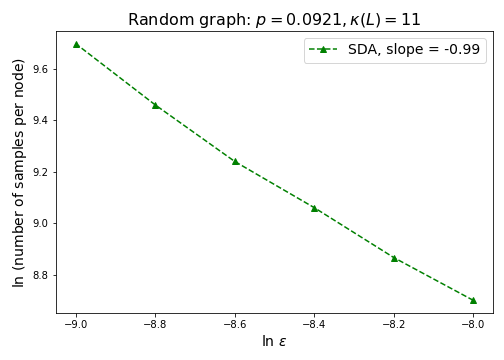}
    \label{fig:sample random graph SDA}
\end{subfigure}
\caption{Sample complexity of \texttt{SDA} for smaller $\varepsilon$ on different graphs with $100$ nodes.}
\label{fig:sample complexity smaller epsilon SDA}
\end{figure}

In our second set of experiments, we compare \texttt{SDA} with the state-of-the-art dual accelerated method \texttt{SSTM\_sc} \citep{gorbunov2019optimal} (see comment (4) of Section \ref{sec:main results}) by focusing on sample complexity. Following~\citet{gorbunov2019optimal}, we define the sample complexity of a distributed algorithm as the number of samples per node needed to ensure that $\mathbb{E}\left[\sum^N_{i=1} \norm{\widehat{\theta}_i - \bar{\mu}}^2_2\right] \leq \varepsilon$, where $\widehat{\theta}_i$ is the agent $i$'s estimate of $\bar{\mu}$ returned by a distributed algorithm, and $\varepsilon > 0$ is the given accuracy. The other dual accelerated method \texttt{R-RRMA-AC-SA$^2$} (also mentioned in comment (4) of Section \ref{sec:main results}) has many more tunable parameters than \texttt{SSTM\_sc} and is challenging to implement. Owing to the fact that it has the sample complexity guarantees as \texttt{SSTM\_sc}, we compare our algorithm only with \texttt{SSTM\_sc}.   

We plot the sample complexity as a function of $\varepsilon$ on a log-log plot. As illustrated in Figure \ref{fig:sample complexity large epsilon}, (i) \texttt{SDA} requires considerably smaller number of samples per node than \texttt{SSTM\_sc} to reach $\varepsilon$ accuracy for each graph; (ii) the slopes of \texttt{SSTM\_sc} for different graphs are nearly $-2$, which verifies the $ \mathcal{O}\left(\frac{1}{\varepsilon^2}\right)$ sample complexity of \texttt{SSTM\_sc} alluded to in comment (3) of Section \ref{sec:main results}; (iii) the slopes of \texttt{SDA} vary between $-1$ and $-0.5$ for different graphs. This is because the sample complexity of \texttt{SDA} is determined by two terms: $\mathcal{O}\left(\frac{N^{1/2} \kappa(L)^{1/4} }{\varepsilon^{1/2}} \right)$ and $\mathcal{O}\left(\frac{1}{\varepsilon}\right)$, and for the values of $\varepsilon$ in Figure \ref{fig:sample complexity large epsilon}, neither term is dominant. Therefore, we perform a complementary set of experiments only with \texttt{SDA} for much smaller $\varepsilon$ so that $\mathcal{O}\left(\frac{1}{\varepsilon}\right)$ dominates $\mathcal{O}\left(\frac{N^{1/2} \kappa(L)^{1/4}}{\varepsilon^{1/2}} \right)$. We can see from Figure \ref{fig:sample complexity smaller epsilon SDA} that the slope of \texttt{SDA} is approximately $-1$ for each graph, which clearly demonstrates the $\mathcal{O}\left(\frac{1}{\varepsilon}\right)$ sample complexity of \texttt{SDA}.

\subsection{Non-asymptotic behavior: \texttt{SDA} vs \texttt{D-MASG}}
\label{sec:numerical results non-asymptotic regime behavior}
\begin{figure}[ht!]
\centering
\includegraphics[width=0.5\textwidth]{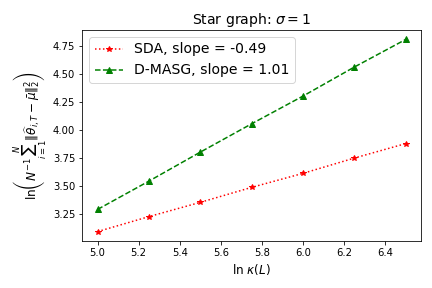}
\caption{Non-asymptotic behavior of \texttt{SDA} and \texttt{D-MASG} on a series of star graphs.}
\label{fig:non-asymptotic behavior}
\end{figure}

In our third experiment, we investigate the performance of \texttt{SDA} and \texttt{D-MASG} in the non-asymptotic regime where $T \asymp \sqrt{\kappa(L)}$. As shown in comment (3) of Section \ref{sec:main results}, the deterministic error of both methods decays linearly at an accelerated rate with exponent $\mathcal{O}\left(- \frac{T}{\sqrt{\kappa(L)}}\right)$ and the stochastic error of both methods has asymptotic decay rate $\mathcal{O}\left(\frac{\sigma^2}{T}\right)$. Therefore, the non-asymptotic behavior of both methods when $T \asymp \sqrt{\kappa(L)}$ should tell us which algorithm has a smaller higher-order term in its stochastic error and hence enjoys better finite-time convergence behavior. 
We consider the star graph for this experiment because the network connectivity parameter $\kappa(L)$ of the star graph with $N$ nodes is approximately equal to $N$ given that we use the Metropolis--Hastings weight matrix to construct $L$ (see Appendix \ref{sec:implementation details} for details). Thus, we are able to generate a family of star graphs with ascending $\kappa(L)$ by simply increasing $N$. 
We plot $\ln \mathbb{E}\left[\sum^N_{i=1} \norm{\widehat{\theta}_{i,T} - \bar{\mu}}^2_2\right]$ as a function of $\ln \kappa(L)$, where $T$ is set to $\sqrt{\kappa(L)}$ for each star graph instance. As shown in Figure \ref{fig:non-asymptotic behavior}, the slope of \texttt{SDA} is very close to 0.5 and the slope of \texttt{D-MASG} is around 1, which demonstrates the superior performance of \texttt{SDA} in the given finite-time regime. This behavior was alluded to in comment (3) of Section \ref{sec:main results}: When $T \asymp \sqrt{\kappa(L)}$, \texttt{SDA} enjoys the error bound $\mathcal{O}\left(\sqrt{\kappa(L)}\sigma^2\right)$ while \texttt{D-MASG} only attains the $\mathcal{O}\left(\kappa(L)\sigma^2\right)$ error.

%% file: proof-of-main-results.tex
\section{Proofs of Main Results}
This section is devoted to the proofs of our main results---Theorem \ref{theorem:upper bound of stochastic dual accelerated method} and its associated corollaries. We use several technical lemmas in these proofs, which are stated and proved in Section~\ref{sec:tech-lemmas} to follow.

\subsection{Proof of Theorem \ref{theorem:upper bound of stochastic dual accelerated method}}
\label{sec:proof of theorem 1}

Since the system in \texttt{SDA} consisting of $\{\left(\theta_{i,t}, x_{i,t}, y_{i,t}\right) |~ i \in \mathcal{N} \text{ and } t \geq 0\}$ is linear, the evolutions of different columns of vectors $\theta_{i,t}, x_{i,t}$ and $y_{i,t}$ along $N$ nodes are independent. Thus, it is sufficient to fix an arbitrary column $j \in \{1,\ldots,n\}$ for analysis at the moment. \texttt{SDA} repeats for all integer $t \geq 0$:
\begin{subequations}
\label{DA1}
\begin{align}
\widetilde{\Theta}_{t,j} &= \widetilde{X}_{t,j} + \widetilde{R}_{t,j}, \label{DAa}\\
\widetilde{Y}_{t+1,j} &= \widetilde{X}_{t,j} - \eta  L \widetilde{\Theta}_{t,j}, \label{DAb}\\
\widetilde{X}_{t+1,j} &= \widetilde{Y}_{t+1,j} + \zeta(\widetilde{Y}_{t+1,j} - \widetilde{Y}_{t,j}) \label{DAc},
\end{align}
\end{subequations}
where $\widetilde{\Theta}_{t,j} = \begin{bmatrix}
\theta_{1,t,j} \cdots \theta_{N,t,j}
\end{bmatrix}^\top$, $\widetilde{X}_{t,j} = \begin{bmatrix}
x_{1,t,j} \cdots x_{N,t,j} 
\end{bmatrix}^\top$, $\widetilde{Y}_{t,j} = \begin{bmatrix}
y_{1,t,j} \cdots y_{N,t,j} 
\end{bmatrix}^\top$ and $\widetilde{R}_{t,j} = \begin{bmatrix}
R_{1,t,j} \cdots R_{N,t,j} 
\end{bmatrix}^\top$. Here, we use $\theta_{i,t,j}, x_{i,t,j}, y_{i,t,j}$ and $R_{i,t,j}$ to denote the $j$-th component of the vectors $\theta_{i,t}, x_{i,t}, y_{i,t}$ and $R_{i,t}$, respectively. Note that $\widetilde{R}_{t,j}$ has mean vector $\widetilde{\mu}_j := \begin{bmatrix}
\mu_{1,j} \cdots \mu_{N,j} \end{bmatrix}^\top \in \mathbb{R}^N$ and diagonal covariance matrix $\widetilde{\Sigma}_j :=\text{diag}(\sigma^2_{1,j},\ldots,\sigma^2_{N,j}) \in \mathbb{R}^{N \times N}$. Let $U_{t,j} := \widetilde{X}_{t,j} - \left(\bar{\mu}_j e_N - \widetilde{\mu}_j\right)$, where $e_N \in \mathbb{R}^N$ is the all-ones vector, and $V_{t,j} := \widetilde{Y}_{t,j} - \widetilde{Y}_{t-1,j}$. Then equations \eqref{DA1} can be rewritten as
\begin{subequations}
\label{DA2}
\begin{align}
\widetilde{\Theta}_{t,j} - \bar{\mu}_j e_N &= U_{t,j} + \omega_{t,j}, \label{DAd}\\
U_{t+1,j} &= \Big[I_{N} - \eta (1+\zeta) L \Big] U_{t,j} + \zeta^2 V_{t,j} - \eta (1+\zeta) L\omega_{t,j}, \label{DAe}\\
V_{t+1,j} &= -\eta L U_{t,j} + \zeta V_{t,j} - \eta L \omega_{t,j}, \label{DAf}
\end{align}
\end{subequations}
where $\omega_{t,j} := \widetilde{R}_{t,j} - \widetilde{\mu}_j$ has zero mean and covariance matrix $\widetilde{\Sigma}_j$. Combining Eq.~\eqref{DAe} and \eqref{DAf}, we have the following recursion for all integer $t\geq 0$:
\begin{align}
\label{DA recursion}
\begin{bmatrix}
U_{t+1,j}\\
V_{t+1,j}
\end{bmatrix} &= A \begin{bmatrix}
U_{t,j}\\
V_{t,j}
\end{bmatrix} - \eta B\omega_{t,j},
\end{align}
where
\begin{align}
\label{A and B}
A:= \begin{bmatrix}
I_{N} - \eta (1+\zeta) L & \zeta^2  I_{N}\\
-\eta L & \zeta I_{N} \\
\end{bmatrix} ~\text{ and }~ B:= \begin{bmatrix}
(1+\zeta)  L \\
L
\end{bmatrix}.
\end{align}
Unrolling the recursion \eqref{DA recursion}, we obtain that for all integer $t \geq 1$:
\begin{align}
\label{DA recursion unrolled}
\begin{bmatrix}
U_{t,j}\\
V_{t,j}
\end{bmatrix} &= A^t \begin{bmatrix}
U_{0,j}\\
V_{0,j}
\end{bmatrix} - \eta \sum^{t-1}_{k=0} A^{t-1-k} B \omega_{k,j},
\end{align}
where $U_{0,j} = \widetilde{X}_{0,j} - \left(\bar{\mu}_j e_N - \widetilde{\mu}_j \right) = \widetilde{\mu}_j - \bar{\mu}_j e_N$ and $V_{0,j} = 0$. Noting that
$\widehat{\Theta}_{T,j} := \begin{bmatrix}
\widehat{\theta}_{1,T,j} \cdots \widehat{\theta}_{N,T,j}
\end{bmatrix}^\top = \frac{1}{T-T_0} \sum^{T-1}_{t=T_0} \widetilde{\Theta}_{t, j}$, we have
\begin{align*}
\widehat{\Theta}_{T,j} - \bar{\mu}_j e_N &= \frac{1}{T-T_0} \sum^{T-1}_{t=T_0} \left(\widetilde{\Theta}_{t,j} - \bar{\mu}_j e_N\right) \overset{(a)}{=}\frac{1}{T-T_0} \sum^{T-1}_{t=T_0} U_{t,j} + \frac{1}{T-T_0} \sum^{T-1}_{t=T_0} \omega_{t,j}\\
&\overset{(b)}{=} \frac{1}{T-T_0} \begin{bmatrix}
I_{N} & 0
\end{bmatrix}\left(\sum^{T-1}_{t=T_0} A^t\right) \begin{bmatrix}
U_{0,j}\\
V_{0,j}
\end{bmatrix} - \frac{\eta}{T-T_0} \begin{bmatrix}
I_{N} & 0
\end{bmatrix}\sum^{T-1}_{t=T_0}\sum^{t-1}_{k=0} A^{t-1-k}B \omega_{k,j} \\
&\quad+ \frac{1}{T-T_0} \sum^{T-1}_{t=T_0} \omega_{t,j}\\
&\overset{(c)}{=} \frac{1}{T-T_0} \begin{bmatrix}
I_{N} & 0
\end{bmatrix} \left(\sum^{T-1}_{t=T_0} A^t\right) \begin{bmatrix}
U_{0,j}\\
V_{0,j}
\end{bmatrix}\\
&\quad- \frac{\eta}{T-T_0}\sum^{T_0-1}_{t=0} \begin{bmatrix}
I_N & 0
\end{bmatrix}A^{T_0-1-t} \left(\sum^{T-T_0-1}_{k=0}A^{k}\right) B \omega_{t,j}\\
&\quad- \frac{\eta}{T-T_0} \sum^{T-2}_{t=T_0} \left[ \begin{bmatrix}
I_N & 0
\end{bmatrix} \left(\sum^{T-2-t}_{k=0}A^{k}\right)B - I_N \right] \omega_{t,j} +\frac{1}{T-T_0} \omega_{T-1,j}.
\end{align*}
Here step (a) is due to Eq.~\eqref{DAd}, step (b) follows from $U_{t,j}= \begin{bmatrix}
I_N & 0
\end{bmatrix}\begin{bmatrix}
U_{t,j}\\
V_{t,j}
\end{bmatrix}$ and Eq.~\eqref{DA recursion unrolled}, and step (c) uses the following identity:
\begin{align*}
\sum^{T-1}_{t=T_0}\sum^{t-1}_{k=0}A^{t-1-k}B \omega_{k,j} &= \sum^{T_0-1}_{t=0}A^{T_0-1-t} \left(\sum^{T-T_0-1}_{k=0}A^{k}\right) B \omega_{t,j} +\sum^{T-2}_{t=T_0}\left(\sum^{T-2-t}_{k=0}A^{k}\right)B \omega_{t,j}.
\end{align*}
Thus, by the independence of the random vectors $\{\omega_{t, j}\}^{T-1}_{t=0}$, we have
\begin{align}
\label{eq:mean squared error decomposition}
\begin{split}
\mathbb{E}\left[\norm{\widehat{\Theta}_{T,j} - \bar{\mu}_j e_N}^2_2\right]
=& \frac{1}{(T-T_0)^2} \underbrace{\norm{\begin{bmatrix}
I_N & 0
\end{bmatrix} \left(\sum^{T-1}_{t=T_0} A^t \right)\begin{bmatrix}
U_{0,j}\\
V_{0,j}
\end{bmatrix}}^2_2}_{\Delta_1}\\
&+ \frac{1}{(T-T_0)^2} \underbrace{\sum^{T_0-1}_{t=0} \norm{\eta \begin{bmatrix}
I_N & 0
\end{bmatrix}A^{T_0-1-t} \left(\sum^{T-T_0-1}_{k=0}A^{k}\right) B \sqrt{\tilde{\Sigma}_j}}^2_F}_{\Delta_2}\\
&+ \frac{1}{(T-T_0)^2} \underbrace{\sum^{T-2}_{t=T_0} \norm{\left[\eta \begin{bmatrix}
I_N & 0
\end{bmatrix}\left( \sum^{T-2-t}_{k=0}A^{k}\right) B - I_N\right]\sqrt{\tilde{\Sigma}_j}}^2_F}_{\Delta_3}\\
&+ \frac{1}{(T-T_0)^2} \sum^N_{i=1} \sigma^2_{i,j},
\end{split} 
\end{align}
where $\norm{\cdot}_F$ is the Frobenius norm. Next, we bound $\{\Delta_j\}_{j = 1}^3$ in turn.

\paragraph{Bounding $\Delta_1$:}
By Lemma \ref{lemma: permutation and block diagonal matrix}, we have for any integer $t\geq 0$:
\begin{align}
\label{eq:bound (1)6}
A^t  = \begin{bmatrix}
Q  & 0\\
0  & Q \\
\end{bmatrix} P_{\pi}^\top \begin{bmatrix}
A(\lambda_1(L))^t & &\\
 & \ddots & \\
 & & A(\lambda_{N}(L))^t
\end{bmatrix} P_{\pi} \begin{bmatrix}
Q^\top  & 0\\
0 & Q^\top \\
\end{bmatrix},
\end{align}
where $P_{\pi}$ is a permutation matrix.
Since $V_{0,j} = 0$, we have
\begin{align}
\label{eq:bound (1)1}
\begin{bmatrix}
I_N & 0
\end{bmatrix} A^t \begin{bmatrix}
U_{0,j}\\
V_{0,j}
\end{bmatrix} = \begin{bmatrix}
Q  & 0\\
\end{bmatrix} P_{\pi}^\top \begin{bmatrix}
A(\lambda_1(L))^t & & \\
 & \ddots & \\
 & & A(\lambda_{N}(L))^t
\end{bmatrix} P_{\pi} \begin{bmatrix}
\tilde{U}_{0,j}\\
0\\
\end{bmatrix},
\end{align}
where $\widetilde{U}_{j} := Q^\top U_{0,j}$. Since $U_{0,j} = \widetilde{\mu}_j - \bar{\mu}_j e_N$ is orthogonal to $e_N$, which is a scalar multiple of the last column of $Q$, the last element of  $\widetilde{U}_{j}$ is equal to zero, i.e., $\widetilde{U}_{j,N} = 0$. As a result, we have
\begin{align}
\label{eq:bound (1)3}
\sum^{N-1}_{i=1} \widetilde{U}^2_{j,i} = \sum^{N}_{i=1} \widetilde{U}^2_{j,i} = \norm{\widetilde{U}_j}^2_2 = \norm{U_{0,j}}^2_2 = \norm{\widetilde{\mu}_j - \bar{\mu}_j e_N}^2_2.
\end{align}
By the definition of the permutation matrix $P_\pi$ (see Lemma \ref{lemma: permutation and block diagonal matrix}), we have
\begin{align*}
P_{\pi} \begin{bmatrix}
\widetilde{U}_j\\
0\\
\end{bmatrix} = \begin{bmatrix}
\widetilde{U}_{j,1}\\
0\\
\vdots\\
\widetilde{U}_{j,N}\\
0\\
\end{bmatrix},
\end{align*}
from which it follows that
\begin{align*}
\begin{bmatrix}
A(\lambda_1(L))^t & & \\
 & \ddots & \\
 & & A(\lambda_{N}(L))^t
\end{bmatrix} P_{\pi} \begin{bmatrix}
\widetilde{U}_j\\
0\\
\end{bmatrix} = \begin{bmatrix}
\widetilde{U}_{j,1} \cdot [A(\lambda_1(L))^t]_{1,1}\\
\widetilde{U}_{j,1} \cdot [A(\lambda_1(L))^t]_{2,1}\\
\vdots\\
\widetilde{U}_{j,N} \cdot [A(\lambda_N(L))^t]_{1,1}\\
\widetilde{U}_{j,N} \cdot [A(\lambda_N(L))^t]_{2,1}\\
\end{bmatrix},
\end{align*}
where $[A(\lambda_{i}(L))^t]_{k,l}$ is the $(k,l)$-th element of the matrix $A(\lambda_{i}(L))^t$. 
Since $P^\top_\pi$ is the inverse of $P_\pi$, we have
\begin{align*}
P_{\pi}^\top \begin{bmatrix}
A(\lambda_1(L))^t & & \\
 & \ddots & \\
 & & A(\lambda_{N}(L))^t
\end{bmatrix} P_{\pi} \begin{bmatrix}
\widetilde{U}_j\\
0\\
\end{bmatrix} = \begin{bmatrix}
\widetilde{U}_{j,1} \cdot [A(\lambda_1(L))^t]_{1,1}\\
\vdots\\
\widetilde{U}_{j,N} \cdot [A(\lambda_N(L))^t]_{1,1}\\
\widetilde{U}_{j,1} \cdot [A(\lambda_1(L))^t]_{2,1}\\
\vdots\\
\widetilde{U}_{j,N} \cdot [A(\lambda_N(L))^t]_{2,1}\\
\end{bmatrix},
\end{align*}
which yields
\begin{align}
\label{eq:bound (1)2}
\begin{split}
\begin{bmatrix}
Q  & 0\\
\end{bmatrix} P_{\pi}^\top \begin{bmatrix}
A(\lambda_1(L))^t & & \\
 & \ddots & \\
 & & A(\lambda_{N}(L))^t
\end{bmatrix} P_{\pi} \begin{bmatrix}
\widetilde{U}_j\\
0\\
\end{bmatrix} = Q \begin{bmatrix}
\widetilde{U}_{j,1} \cdot [A(\lambda_1(L))^t]_{1,1}\\
\vdots\\
\widetilde{U}_{j,N} \cdot [A(\lambda_N(L))^t]_{1,1}\\
\end{bmatrix}.
\end{split} 
\end{align}
Combining Eq.~\eqref{eq:bound (1)1} and \eqref{eq:bound (1)2}, we have
\begin{align*}
\begin{bmatrix}
I_N & 0
\end{bmatrix} \left(\sum^{T-1}_{t=T_0} A^t \right)\begin{bmatrix}
U_{0,j}\\
V_{0,j}
\end{bmatrix} = \sum^{T-1}_{t=T_0} \begin{bmatrix}
I_N & 0
\end{bmatrix} A^t \begin{bmatrix}
U_{0,j}\\
V_{0,j}
\end{bmatrix} = Q \begin{bmatrix}
\widetilde{U}_{j,1} \cdot \sum^{T-1}_{t=T_0}[A(\lambda_1(L))^t]_{1,1}\\
\vdots\\
\widetilde{U}_{j,N} \cdot \sum^{T-1}_{t=T_0}[A(\lambda_N(L))^t]_{1,1}\\
\end{bmatrix},
\end{align*}
which implies that
\begin{align}
\label{eq:bound (1)4}
\begin{split}
\Delta_1 &= \sum^{N}_{i=1} \left(\widetilde{U}_{j,i} \sum^{T-1}_{t=T_0}[A(\lambda_i(L))^t]_{1,1} \right)^2 \\
&= \sum^{N-1}_{i=1}\left[ \widetilde{U}_{j,i}^2 \cdot \left(\sum^{T-1}_{t=T_0}[A(\lambda_i(L))^t]_{1,1} \right)^2 \right] \\
&\overset{(a)}{\leq} \left(\sum^{N-1}_{i=1} \widetilde{U}_{j,i}^2\right) \left[\sum^{T-1}_{t=T_0} \left(1 - \frac{1}{2\sqrt{\kappa(L)}}\right)^t \right]^2 \\
&\overset{(b)}{=} \norm{\widetilde{\mu}_j - \bar{\mu}_j e_N}^2_2 \left(1 - \frac{1}{2\sqrt{\kappa(L)}}\right)^{2T_0} \left[\sum^{T-T_0-1}_{t=0} \left(1 - \frac{1}{2\sqrt{\kappa(L)}}\right)^t \right]^2 \\
&\overset{(c)}{\leq} 4\kappa(L) e^{-\frac{T}{2\sqrt{\kappa(L)}}}\norm{\widetilde{\mu}_j - \bar{\mu}_j e_N}^2_2.
\end{split}
\end{align}
Here step (a) uses part (a) of Lemma \ref{lemma:upper bounds related to A(lambda_i)}, step (b) follows from Eq. \eqref{eq:bound (1)3}, and step (c) follows from the elementary bounds
\begin{align*}
\left(1 - \frac{1}{2\sqrt{\kappa(L)}}\right)^{2T_0} \leq e^{-\frac{T}{2\sqrt{\kappa(L)}}} \text{ and } \sum^{T-T_0-1}_{t=0} \left(1 - \frac{1}{2\sqrt{\kappa(L)}}\right)^t \leq \sum^{\infty}_{t=0} \left(1 - \frac{1}{2\sqrt{\kappa(L)}}\right)^t = 2\sqrt{\kappa(L)}.
\end{align*}

\paragraph{Bounding $\Delta_2$:}
By Lemma \ref{lemma: permutation and block diagonal matrix}, we have
\begin{align}
\label{eq:bound (2)4}
B = \begin{bmatrix}
Q  & 0\\
0  & Q \\
\end{bmatrix} P^\top_{\pi} \begin{bmatrix}
B(\lambda_1(L)) & &\\
 & \ddots & \\
 & & B(\lambda_N(L))
\end{bmatrix} Q^\top.
\end{align}
Combining this with Eq.~\eqref{eq:bound (1)6} yields
\begin{align*}
\eta\begin{bmatrix}
I_N & 0
\end{bmatrix}A^{T_0-1-t} \left(\sum^{T-T_0-1}_{k=0}A^{k}\right) B = \begin{bmatrix}
Q  & 0\\
\end{bmatrix} P_{\pi}^\top \begin{bmatrix}
C(\lambda_1(L)) & &\\
 & \ddots & \\
 & & C(\lambda_N(L))
\end{bmatrix} Q^\top,\\
\end{align*}
where 
\begin{align*}
C(\lambda_i(L), t) := A(\lambda_i(L))^{T_0-1-t} \eta \left(\sum^{T-T_0-1}_{k=0} A(\lambda_i(L))^k \right)B(\lambda_i(L)).
\end{align*}
It is easy to see that $C(\lambda_N(L), t) = \begin{bmatrix}
0\\
0\\
\end{bmatrix}$ for all $t$ because $B(\lambda_N(L)) = \begin{bmatrix}
0\\
0\\
\end{bmatrix}$. Since the spectral radius of $A(\lambda_i(L))$ is upper bounded by $1-\frac{1}{\sqrt{\kappa(L)}} < 1$ for all $i=1,\ldots,N-1$, we have for all $t=0, \ldots, T_0-1$:
\begin{align*}
\eta \left(\sum^{T-T_0-1}_{k=0} A(\lambda_i(L))^k \right) &= \eta \Big(I - A(\lambda_i(L))^{T-T_0}\Big) \Big(I - A(\lambda_i(L))\Big)^{-1}\\
&= \Big(I - A(\lambda_i(L))^{T-T_0}\Big) \frac{1}{ \lambda_i(L)}\begin{bmatrix}
1-\zeta & \zeta^2\\
-\eta \lambda_i(L) & \eta (1+\zeta) \lambda_i(L)\\
\end{bmatrix}.
\end{align*}
Hence, we have for all $i = 1, \ldots, N-1$:
\begin{align*}
&C(\lambda_i(L), t)\\
=& A(\lambda_i(L))^{T_0-1-t}\Big(I - A(\lambda_i(L))^{T-T_0}\Big) \frac{1}{ \lambda_i(L)}\begin{bmatrix}
1-\zeta & \zeta^2\\
-\eta \lambda_i(L) & \eta (1+\zeta) \lambda_i(L)\\
\end{bmatrix}B(\lambda_i(L))\\
=&A(\lambda_i(L))^{T_0-1-t}\Big(I - A(\lambda_i(L))^{T-T_0}\Big) \begin{bmatrix}
1\\
0\\
\end{bmatrix}\\
=&A(\lambda_i(L))^{T_0-1-t}\begin{bmatrix}
1- [A(\lambda_i(L))^{T-T_0}]_{1,1}\\
-[A(\lambda_i(L))^{T-T_0}]_{2,1}\\
\end{bmatrix}\\
=&\begin{bmatrix}
[A(\lambda_i(L))^{T_0-1-t}]_{1,1}\left(1- [A(\lambda_i(L))^{T-T_0}]_{1,1}\right) - [A(\lambda_i(L))^{T_0-1-j}]_{1,2}[A(\lambda_i(L))^{T-T_0}]_{2,1}\\
[A(\lambda_i(L))^{T_0-1-j}]_{2,1}\left(1- [A(\lambda_i(L))^{T-T_0}]_{1,1}\right) - [A(\lambda_i(L))^{T_0-1-j}]_{2,2}[A(\lambda_i(L))^{T-T_0}]_{2,1}\\
\end{bmatrix}.
\end{align*}
Moreover, we have for all $t=0, \ldots, T_0-1$:
\begin{align*}
& \begin{bmatrix}
Q  & 0\\
\end{bmatrix} P_{\pi}^\top \begin{bmatrix}
C(\lambda_1(L), t) & &\\
 & \ddots & \\
 & & C(\lambda_{N}(L), t)
\end{bmatrix} Q^\top = Q \begin{bmatrix}
C(\lambda_1(L), t)_1 & &\\
 & \ddots & \\
 & & C(\lambda_{N}(L), t)_1
\end{bmatrix} Q^\top,
\end{align*}
where $C(\lambda_{i}(L), t)_1$ denotes the first element of $C(\lambda_{i}(L), t)$ for all $i=1,\ldots,N$. Therefore,
\begin{align} \label{eq:bound (2)3}
\begin{split}
& \norm{\eta\begin{bmatrix}
I_N & 0
\end{bmatrix}A^{T_0-1-t} \left(\sum^{T-T_0-1}_{k=0}A^{k}\right) B \sqrt{\widetilde{\Sigma}_j}}^2_F \\
=& \norm{Q \begin{bmatrix}
C(\lambda_1(L), t)_1 & &\\
 & \ddots & \\
 & & C(\lambda_{N}(L), t)_1
\end{bmatrix} Q^\top \sqrt{\widetilde{\Sigma}_j}}^2_F \\
=& ~\text{tr}\left(\underbrace{Q \begin{bmatrix}
[C(\lambda_1(L), t)_1]^2 & &\\
 & \ddots & \\
 & & [C(\lambda_{N}(L), t)_1]^2
\end{bmatrix} Q^\top}_{C(t)} \widetilde{\Sigma}_j\right).
\end{split}
\end{align}
Note that the matrix $C(t)$ has non-negative eigenvalues $[C(\lambda_1(L), t)_1]^2, \ldots, [C(\lambda_N(L), t)_1]^2$. By part (b) of Lemma \ref{lemma:upper bounds related to A(lambda_i)}, we have all $i=1,\cdots,N-1$ and $t = 0, \cdots, T_0-1$:
\begin{align}
\label{eq:bound (2)1}
\abs{C(\lambda_i(L), t)_1} \leq 2\abs{\left(1+\frac{T_0-1-t}{\sqrt{\kappa(L)}+1}\right)\left(1 - \frac{1}{\sqrt{\kappa(L)}}\right)^{T_0-1-t}}.
\end{align}
Using bound \eqref{eq:bound (2)1} and von Neumann's trace inequality then yields that for all $t = 0, \cdots, T_0-1$:
\begin{equation}
\label{eq:bound (2)2}
\begin{split}
\text{tr}\Big(C(t) \widetilde{\Sigma}_j \Big) \leq 4\left[\left(1+\frac{T_0-1-t}{\sqrt{\kappa(L)}+1}\right)\left(1 - \frac{1}{\sqrt{\kappa(L)}}\right)^{T_0-1-t}\right]^2\sum^{N}_{i=1} \sigma^2_{i,j}.
\end{split}
\end{equation}
Combining the pieces, we have
\begin{align}
\label{eq:bound (2)5}
\begin{split}
\Delta_2 &= \sum^{T_0-1}_{t=0} \text{tr}\Big(C(t) \widetilde{\Sigma}_j \Big) \leq 4\sum^{T_0-1}_{t=0} \left[\left(1+\frac{T_0-1-t}{\sqrt{\kappa(L)}+1}\right)\left(1 - \frac{1}{\sqrt{\kappa(L)}}\right)^{T_0-1-t}\right]^2\sum^{N}_{i=1} \sigma^2_{i,j} \\
&\leq 4\sum^{+\infty}_{k=0} \left[\left(1+\frac{k}{\sqrt{\kappa(L)}+1}\right)\left(1 - \frac{1}{\sqrt{\kappa(L)}}\right)^k\right]^2 \sum^{N}_{i=1} \sigma^2_{i,j} \\
&\leq 4\left( k^* + \sqrt{\kappa(L)}\right)\sum^{N}_{i=1} \sigma^2_{i, j},
\end{split}
\end{align}
where the last inequality is due to Lemma \ref{lemma:infinite sum upper bound}.

\paragraph{Bounding $\Delta_3$:}
It follows from Eq.~\eqref{eq:bound (1)6} and \eqref{eq:bound (2)4} that
\begin{align}
\label{eq:bound (3)1}
\begin{split}
\eta\begin{bmatrix}
I_N & 0
\end{bmatrix} \left(\sum^{T-2-t}_{k=0}A^{k}\right) B &= \begin{bmatrix}
Q  & 0\\
\end{bmatrix} P_{\pi}^\top \begin{bmatrix}
D(\lambda_1(L), t) & &\\
 & \ddots & \\
 & & D(\lambda_N(L), t)
\end{bmatrix} Q^\top \\
&= Q \begin{bmatrix}
D(\lambda_1(L), t)_1 & &\\
 & \ddots & \\
 & & D(\lambda_{N}(L), t)_1
\end{bmatrix} Q^\top,
\end{split}
\end{align}
where
\begin{align*}
D(\lambda_i(L), t) = \eta \left(\sum^{T-2-t}_{k=0} A(\lambda_i(L))^k \right)B(\lambda_i(L)) \quad \text{for all } i =1,\ldots,N,
\end{align*}
and $D(\lambda_{i}(L), t)_1$ denotes the first element of $D(\lambda_{i}(L), t)$. It is easy to observe that $D(\lambda_N(L), t) = \begin{bmatrix}
0\\
0\\
\end{bmatrix}$ for all $t$ because $B(\lambda_N(L)) = \begin{bmatrix}
0\\
0\\
\end{bmatrix}$. In addition, for all $t = T_0, \ldots, T-2$ and $i =1,\cdots,N-1$:
\begin{align}
\label{eq:bound (3)2}
\begin{split}
D(\lambda_i(L), t) &= \eta \Big(I - A(\lambda_i(L))^{T-1-t}\Big) \Big(I - A(\lambda_i(L))\Big)^{-1} B(\lambda_i(L)) \\
&= \begin{bmatrix}
1- [A(\lambda_i(L))^{T-1-t}]_{1,1}\\
-[A(\lambda_i(L))^{T-1-t}]_{2,1}\\
\end{bmatrix}.
\end{split}
\end{align}
Thus, combining Eq. \eqref{eq:bound (3)1} and \eqref{eq:bound (3)2}, the following holds for all $t = T_0, \cdots, T-2$:
\begin{align} \label{eq:bound (3)6}
\begin{split}
& \norm{\left[\eta \begin{bmatrix}
I_N & 0
\end{bmatrix}\left( \sum^{T-2-t}_{k=0}A^{k}\right) B - I_N\right]\sqrt{\widetilde{\Sigma}_j}}^2_F \\
=& \norm{Q \begin{bmatrix}
-[A(\lambda_1(L))^{T-1-t}]_{1,1} & & &\\
 & \ddots & &\\
 & & -[A(\lambda_{N-1})^{T-1-t}]_{1,1} &\\
 & & & -1
\end{bmatrix} Q^\top \sqrt{\widetilde{\Sigma}_j}}^2_F \\
=& ~\text{tr}\left(\underbrace{Q \begin{bmatrix}
[A(\lambda_1(L))^{T-1-t}]^2_{1,1} & & &\\
 & \ddots & &\\
 & & [A(\lambda_{N-1})^{T-1-t}]^2_{1,1} &\\
 & & & 1
\end{bmatrix} Q^\top}_{D(t)} \widetilde{\Sigma}_j \right).
\end{split}
\end{align}
Note that the matrix $D(t)$ has non-negative eigenvalues 
\begin{align*}
[A(\lambda_1(L))^{T-1-t}]^2_{1,1}, \ldots, [A(\lambda_{N-1}(L))^{T-1-t}]^2_{1,1} \text{ and } 1.
\end{align*}
By part (a) of Lemma \ref{lemma:upper bounds related to A(lambda_i)}, we have for all $i=1,\cdots,N-1$ and $t = T_0, \cdots, T-2$:
\begin{align}
\label{eq:bound (3)3}
\abs{[A(\lambda_i(L))^{T-1-t}]_{1,1}} \leq \left(1+\frac{T-1-t}{\sqrt{\kappa(L)}+1}\right)\left(1 - \frac{1}{\sqrt{\kappa(L)}}\right)^{T-1-t} < 1.
\end{align}
Using bound \eqref{eq:bound (3)3} and von Neumann's trace inequality again yields the following bound for all $t = T_0, \cdots, T-2$:
\begin{align}
\label{eq:bound (3)5}
\text{tr}\Big(D(t) \widetilde{\Sigma}_j \Big) \leq \max_{i \in \mathcal{N}} \sigma^2_{i,j} + \left[\left(1+\frac{T-1-t}{\sqrt{\kappa(L)}+1}\right)\left(1 - \frac{1}{\sqrt{\kappa(L)}}\right)^{T-1-t}\right]^2\sum^{N}_{i=1} \sigma^2_{i,j}.
\end{align}
Proceeding as before, we have
\begin{align}
\label{eq:bound (3)7}
\begin{split}
\Delta_3 &= \sum^{T-2}_{t=T_0} \text{tr}\Big(D(t) \widetilde{\Sigma}_j \Big) \\
&\leq \frac{T}{2}\max_{i \in \mathcal{N}} \sigma^2_{i,j} + \sum^{+\infty}_{k=0} \left[\left(1+\frac{k}{\sqrt{\kappa(L)}+1}\right)\left(1 - \frac{1}{\sqrt{\kappa(L)}}\right)^k\right]^2 \sum^{N}_{i=1} \sigma^2_{i,j} \\
&\leq \frac{T}{2}\max_{i \in \mathcal{N}} \sigma^2_{i,j} + \left( k^* + \sqrt{\kappa(L)}\right)\sum^{N}_{i=1} \sigma^2_{i,j},
\end{split}
\end{align}
where the last inequality is due to Lemma \ref{lemma:infinite sum upper bound}.

\paragraph{Putting steps together:}
Plugging bounds \eqref{eq:bound (1)4}, \eqref{eq:bound (2)5} and \eqref{eq:bound (3)7} into Eq.~\eqref{eq:mean squared error decomposition}, we have
\begin{align}
\label{eq:mean squared error bound for each j}
\begin{split}
\mathbb{E}\left[\norm{\widehat{\Theta}_{T,j} - \bar{\mu}_j e_N}^2_2\right]
\leq& \frac{16\kappa(L)}{T^2 e^{\frac{T}{2\sqrt{\kappa(L)}}}}\norm{\widetilde{\mu}_j - \bar{\mu}_j e_N}^2_2 \\
&+ \frac{24 \left( k^* + \sqrt{\kappa(L)} \right) \sum^{N}_{i=1} \sigma^2_{i,j}}{T^2} + \frac{2\max_{i \in \mathcal{N}} \sigma^2_{i,j}}{T}.
\end{split}
\end{align}
Since bound \eqref{eq:mean squared error bound for each j} holds for all $j \in \{1, \ldots, n\}$, $\sum^N_{i=1} \norm{\widehat{\theta}_{i,T} - \bar{\mu}}^2_2 = \sum^n_{j=1} \norm{\widehat{\Theta}_{T,j} - \bar{\mu}_j e_N}^2_2$ and \\$\sum^N_{i=1}\norm{\mu_i - \bar{\mu}}^2_2 = \sum^n_{j=1} \norm{\widetilde{\mu}_j - \bar{\mu}_j e_N}^2_2$, we obtain the desired upper bound by summing up both sides of Eq. \eqref{eq:mean squared error bound for each j} over all $j \in \{1, \ldots, n\}$.
\qed

\subsection{Proof of Corollary \ref{corollary:example I}}
\label{sec:proof of corollary 1}
Since $\sigma^2_{i,j}$ is the variance of the $j$-th component of $R_{i,t} = A^\top_i Y_{i,t}$, by Theorem \ref{theorem:upper bound of stochastic dual accelerated method}, we have
\begin{align}
\label{eq:corollary 1 proof (1)}
\begin{split}
\mathbb{E}\left[\sum^N_{i=1} \norm{\widehat{\mu}_{i,T} - \bar{\mu}}^2_2\right] \leq& \mathcal{O}\left(  e^{-\frac{T}{\sqrt{\kappa(L)}}}\sum^N_{i=1}\norm{\mu_i - \bar{\mu}}^2_2 \right) \\
&+ \mathcal{O}\left(\frac{ \sqrt{\kappa(L)} \sum^{N}_{i=1} \sum^n_{j=1} \sigma^2_{i,j}}{T^2} + \frac{\sum^n_{j=1}\max_{i \in \mathcal{N}}\sigma^2_{i,j}}{T}\right).
\end{split}
\end{align}
Moreover, we have
\begin{align}
\label{eq:upper bound inverse of a matrix for corollary 1 (2)}
\sum^N_{i=1} \norm{\left(\widehat{A}_{i,T'}\right)^{-1} - \bar{A}^{-1}}^2_2 &\overset{(a)}{\leq} \sum^N_{i=1} \frac{\norm{\bar{A}^{-1}}^4_2 \norm{\widehat{A}_{i,T'} - \bar{A}}_2^2}{\left(1 - \norm{\bar{A}^{-1} \left(\widehat{A}_{i,T'} - \bar{A}\right)}_2\right)^2} \notag\\
&\overset{(b)}{\leq} 4\norm{\bar{A}^{-1}}^4_2 \sum^N_{i=1} \norm{\widehat{A}_{i,T'} - \bar{A}}^2_2 \notag\\
&\overset{(c)}{\leq} \mathcal{O}\left(e^{-\frac{T'}{\sqrt{\kappa(L)}}} \norm{\bar{A}^{-1}}^4_2\sum^N_{i=1}\norm{A_i^\top A_i - \bar{A}}^2_2\right),
\end{align}
where step (a) applies Lemma \ref{lemma:upper bound inverse of a matrix} with $\norm{\bar{A}^{-1} \left(\widehat{A}_{i,T'} - \bar{A}\right)}_2 < 1$, step(b) uses the inequality $1 - \norm{\bar{A}^{-1} \left(\widehat{A}_{i,T'} - \bar{A}\right)}_2 \geq \frac{1}{2}$, and step (c) follows from bound \eqref{eq:upper bound inverse of a matrix for corollary 1 (1)}. Since we have the error decomposition
\begin{align*}
\widehat{\theta}_{i,T,T'}- \theta^* &= \left(\widehat{A}_{i,T'}\right)^{-1} \widehat{\mu}_{i,T} - \bar{A}^{-1} \bar{\mu}\\
&= \left(\widehat{A}_{i,T'}\right)^{-1} \widehat{\mu}_{i,T} - \bar{A}^{-1} \widehat{\mu}_{i,T} + \bar{A}^{-1} \widehat{\mu}_{i,T} - \bar{A}^{-1} \bar{\mu}\\
&= \left[\left(\widehat{A}_{i,T'}\right)^{-1} - \bar{A}^{-1}\right] \widehat{\mu}_{i,T} + \bar{A}^{-1} \left(\widehat{\mu}_{i,T} - \bar{\mu} \right),
\end{align*}
we obtain the error bound
\begin{align*}
\norm{\widehat{\theta}_{i,T,T'}- \theta^*}^2_2 &\leq 2\norm{\left(\widehat{A}_{i,T'}\right)^{-1} - \bar{A}^{-1}}^2_2 \cdot \norm{\widehat{\mu}_{i,T}}^2_2 + 2\norm{\bar{A}^{-1}}^2_2 \cdot \norm{\widehat{\mu}_{i,T} - \bar{\mu}}^2_2\\
&\leq 4\norm{\left(\widehat{A}_{i,T'}\right)^{-1} - \bar{A}^{-1}}^2_2 \cdot \norm{\widehat{\mu}_{i,T} - \bar{\mu}}^2_2 \\
&\quad+ 4\norm{\left(\widehat{A}_{i,T'}\right)^{-1} - \bar{A}^{-1}}^2_2 \cdot \norm{\bar{\mu}}^2_2 + 2\norm{\bar{A}^{-1}}^2_2 \cdot \norm{\widehat{\mu}_{i,T} - \bar{\mu}}^2_2.
\end{align*}
This in turn yields
\begin{align}
\label{eq:corollary 1 proof (2)}
\begin{split}
\mathbb{E}\left[\sum^N_{i=1} \norm{\widehat{\theta}_{i,T,T'}- \theta^*}^2_2\right] \leq& 4\sum^N_{i=1} \norm{\left(\widehat{A}_{i,T'}\right)^{-1} - \bar{A}^{-1}}^2_2 \cdot \mathbb{E}\left[\norm{\widehat{\mu}_{i,T} - \bar{\mu}}^2_2\right]\\
&+ 4 \norm{\bar{\mu}}^2_2 \sum^N_{i=1} \norm{\left(\widehat{A}_{i,T'}\right)^{-1} - \bar{A}^{-1}}^2_2 + 2\norm{\bar{A}^{-1}}^2_2 \cdot \mathbb{E}\left[\sum^N_{i=1} \norm{\widehat{\mu}^i_T - \bar{\mu}}^2_2\right].
\end{split}
\end{align}
Substituting bounds \eqref{eq:corollary 1 proof (1)} and \eqref{eq:upper bound inverse of a matrix for corollary 1 (2)} into Eq.~\eqref{eq:corollary 1 proof (2)}, we obtain the desired result.
\qed

\subsection{Proof of Corollary \ref{corollary:example II}}
\label{sec:proof of corollary 2}
Recall that the ``plug-in''estimator is given by $\widehat{J}_{i,T} := \left(I - \gamma \widehat{P}_T \right)^{-1} \widehat{r}_{i,T}$ for all $i \in \mathcal{N}$. We have the following error decomposition:
\begin{align}
\label{eq:corollary 2 proof (1)}
\begin{split}
\widehat{J}_{i,T} - J^*  &= \left(I - \gamma \widehat{P}_{T} \right)^{-1} \widehat{r}_{i,T} - J^* \\
&= \left(I - \gamma \widehat{P}_{T} \right)^{-1} \left(\widehat{r}_{i,T} - \bar{r} + \bar{r} \right) - J^* \\
&= \left(I - \gamma \widehat{P}_{T} \right)^{-1} \left(\widehat{r}_{i,T} - \bar{r} \right) + \left[\left(I - \gamma \widehat{P}_{T} \right)^{-1} \bar{r} - J^* \right],
\end{split}
\end{align}
from which it follows that
\begin{align}
\label{eq:corollary 2 proof (5)}
\mathbb{E}\left[\max_{i \in \mathcal{N}}\norm{\widehat{J}_{i,T} - {J^*}}_{\infty} \right] \leq \mathbb{E}\left[\max_{i \in \mathcal{N}}\norm{\left(I - \gamma \widehat{P}_{T} \right)^{-1} \left(\widehat{r}_{i,T} - \bar{r} \right)}_{\infty}\right] + \mathbb{E}\left[\norm{\left(I - \gamma \widehat{P}_{T} \right)^{-1} \bar{r} - J^*}_{\infty}\right].
\end{align}
Now we bound each term on the right-hand side of Eq.~\eqref{eq:corollary 2 proof (5)}. First, we have
\begin{align}
\label{eq:corollary 2 proof (2)}
\begin{split}
\mathbb{E}\left[\max_{i \in \mathcal{N}}\norm{\left(I - \gamma \widehat{P}_{T} \right)^{-1} \left(\widehat{r}_{i,T} - \bar{r} \right)}_{\infty}\right] &\overset{(a)}{\leq} \frac{1}{1-\gamma} \mathbb{E}\left[\max_{i \in \mathcal{N}}\norm{\widehat{r}_{i,T} - \bar{r}}_{\infty}\right] \\
&\leq \frac{1}{1-\gamma} \mathbb{E}\left[\max_{i \in \mathcal{N}}\norm{\widehat{r}_{i,T} - \bar{r}}_{2}\right] \\
&\leq \frac{1}{1-\gamma} \mathbb{E}\left[\sqrt{\sum^N_{i=1} \norm{\widehat{r}_{i,T} - \bar{r}}^2_{2}}\right] \\
&\overset{(b)}{\leq} \frac{1}{1-\gamma} \sqrt{\mathbb{E}\left[\sum^N_{i=1} \norm{\widehat{r}_{i,T} - \bar{r}}^2_{2}\right]},
\end{split}
\end{align}
where step (a) applies Lemma 2.3 of \citet{agarwal2019reinforcement}, and step (b) uses Jensen's inequality. Applying Theorem \ref{theorem:upper bound of stochastic dual accelerated method} to $\mathbb{E}\left[\sum^N_{i=1} \norm{\widehat{r}_{i,T} - \bar{r}}^2_{2}\right]$ then yields
\begin{align}
\label{eq:corollary 2 proof (3)}
\begin{split}
\mathbb{E}\left[\max_{i \in \mathcal{N}}\norm{\left(I - \gamma \widehat{P}_{T} \right)^{-1} \left(\widehat{r}_{i,T} - \bar{r} \right)}_{\infty}\right] & \leq \mathcal{O}\left( e^{-\frac{T}{2\sqrt{\kappa(L)}}} \frac{\sqrt{\sum^N_{i=1}\norm{r_i - \bar{r}}^2_2}}{1 - \gamma}  \right. \\
& \left.+ \frac{\kappa(L)^{1/4} \sqrt{\sum^{N}_{i=1} \sum^n_{j=1} \sigma^2_{i,j}}}{(1 - \gamma)T} + \frac{\sqrt{\sum^n_{j=1}\max_{i \in \mathcal{N}} \sigma^2_{i,j}}}{(1-\gamma)\sqrt{T}}\right).
\end{split}
\end{align}
Next, applying\footnote{Note that the high-probability bound 5(b) from their Theorem 1(b) can be integrated to obtain an analogous bound in expectation.} Theorem 1(b) from \citet{pananjady2020instance}, we have
\begin{align}
\label{eq:corollary 2 proof (4)}
\mathbb{E}\left[\norm{\left(I - \gamma \widehat{P}_{T} \right)^{-1} \bar{r} - J^*}_{\infty}\right] \leq \mathcal{O}\left(\frac{\sqrt{\log n}}{(1 - \gamma)^{3/2} \sqrt{T}} + \frac{\log n}{(1 - \gamma)^2 T} \right).
\end{align}
Finally, substituting bounds \eqref{eq:corollary 2 proof (3)} and \eqref{eq:corollary 2 proof (4)} into Eq.~\eqref{eq:corollary 2 proof (5)} completes the proof. \qed

%% file: technical-lemmas.tex
\section{Technical Lemmas} \label{sec:tech-lemmas}
In this section, we state and prove the lemmas used above.
\begin{lemma}
\label{lemma: permutation and block diagonal matrix}
Let $P_{\pi} \in \{0,1\}^{2N \times 2N}$ be the permutation matrix corresponding to the permutation $\pi: \{1,\cdots,2N\} \rightarrow \{1,\cdots,2N\}$ defined by
\begin{align*}
\pi(i) = \begin{cases}
			\frac{i-1}{2} + 1 & \text{if $i$ is odd},\\
            N+\frac{i}{2} & \text{otherwise}.
		 \end{cases}
\end{align*}
Then, we have
\begin{align*}
A:=\begin{bmatrix}
I_{N} - \eta (1+\zeta) L & \zeta^2  I_{N}\\
-\eta L & \zeta I_{N} \\
\end{bmatrix} = \begin{bmatrix}
Q  & 0\\
0  & Q \\
\end{bmatrix} P^\top_{\pi} \begin{bmatrix}
A(\lambda_1(L)) & &\\
 & \ddots & \\
 & & A(\lambda_{N}(L))
\end{bmatrix} P_{\pi} \begin{bmatrix}
Q^\top  & 0\\
0 & Q^\top \\
\end{bmatrix},
\end{align*}
and
\begin{align*}
B := \begin{bmatrix}
(1+\zeta)  L \\
L
\end{bmatrix} = \begin{bmatrix}
Q  & 0\\
0  & Q \\
\end{bmatrix} P^\top_{\pi} \begin{bmatrix}
B(\lambda_1(L)) & & \\
 & \ddots & \\
 & & B(\lambda_{N}(L))
\end{bmatrix} Q^\top.
\end{align*}
Here
\begin{align*}
A(\lambda_i(L)) := \begin{bmatrix}
1-\eta (1+\zeta) \lambda_i(L) & \zeta^2\\
-\eta \lambda_i(L) & \zeta\\
\end{bmatrix} ~\text{ and }~ B(\lambda_i(L)) := \begin{bmatrix}
(1+\zeta)\lambda_i(L)\\
\lambda_i(L)\\
\end{bmatrix}\quad \text{for } i = 1,\cdots,N.
\end{align*}
\end{lemma}
\begin{proof}
Since $L = Q \Lambda Q^\top$, we have
\begin{align}
\label{eq:lemmq1 1}
A= \begin{bmatrix}
Q  & 0\\
0  & Q \\
\end{bmatrix}\begin{bmatrix}
I_N - \eta (1+\zeta) \Lambda  & \zeta^2 I_N\\
-\eta \Lambda  & \zeta I_N \\
\end{bmatrix}\begin{bmatrix}
Q^\top  & 0\\
0 & Q^\top \\
\end{bmatrix}.
\end{align}
and
\begin{align}
\label{eq:lemmq1 2}
B = \begin{bmatrix}
Q  & 0\\
0  & Q \\
\end{bmatrix} \begin{bmatrix}
(1+\zeta) \Lambda\\
\Lambda\\
\end{bmatrix} Q^\top.
\end{align}
Since the permutation matrix $P_{\pi} \in \{0,1\}^{2N \times 2N}$ is obtained by permuting the rows of the identity matrix $I_{2N}$ according to $\pi$; that is, each entry of $P_\pi$ is given by
\begin{align*}
P_{\pi}(i,j) = \begin{cases}
			1 & \text{if } j=\pi(i),\\
            0 & \text{otherwise,}
		 \end{cases} \quad \text{for } i, j =1,\ldots,N, 
\end{align*}
one can verify by straightforward calculations that
\begin{align}
\label{eq:lemmq1 3}
\begin{bmatrix}
I_N - \eta (1+\zeta) \Lambda  & \zeta^2 I_N\\
-\eta \Lambda  & \zeta I_N \\
\end{bmatrix} = P_{\pi}^\top \begin{bmatrix}
A(\lambda_1(L)) & &\\
 & \ddots & \\
 & & A(\lambda_{N}(L))
\end{bmatrix} P_{\pi} ,
\end{align}
and
\begin{align}
\label{eq:lemmq1 4}
\begin{bmatrix}
(1+\zeta) \Lambda\\
\Lambda\\
\end{bmatrix} = P_{\pi}^\top \begin{bmatrix}
B(\lambda_1(L)) & & \\
 & \ddots & \\
 & & B(\lambda_{N}(L))
\end{bmatrix}.
\end{align}
Finally, substituting Eq.~\eqref{eq:lemmq1 3} into Eq.~\eqref{eq:lemmq1 1}, and Eq.~\eqref{eq:lemmq1 4} into Eq.~\eqref{eq:lemmq1 2} completes the proof.
\end{proof}

\begin{lemma}
\label{lemma:upper bounds related to A(lambda_i)}
Let $\eta = \frac{1}{\lambda_1(L)}$ and $\zeta = \frac{\sqrt{\kappa(L)} - 1}{\sqrt{\kappa(L)} + 1}$. For each $i=1,\cdots,N-1$, we have for any integers $j \geq k \geq 0$:
\begin{itemize}
    \item[(a)] \begin{align}
    \label{eq:lemmq2 (a)}
\abs{[A(\lambda_i(L))^k]_{1,1}} \leq \left(1+\frac{k}{\sqrt{\kappa(L)}+1}\right)\left(1 - \frac{1}{\sqrt{\kappa(L)}}\right)^k,
\end{align}
    \item[(b)] \begin{align}
    \label{eq:lemmq2 (b)}
\abs{[A(\lambda_i(L))^k]_{1,2} \cdot [A(\lambda_i(L))^j]_{2,1}} \leq \left(1+\frac{k}{\sqrt{\kappa(L)}+1}\right)\left(1 - \frac{1}{\sqrt{\kappa(L)}}\right)^k,
\end{align}
\end{itemize}
where $A(\lambda_i(L))$ is defined in Lemma \ref{lemma: permutation and block diagonal matrix}, and $[A(\lambda_i(L))^k]_{l,m}$ denotes the $(l,m)$-th element of the matrix $A(\lambda_i(L))^k$.
\end{lemma}
\begin{proof}
For each $i=1,\ldots,N-1$, the characteristic polynomial of $A(\lambda_i(L))$ is given by
\begin{align*}
p_i(\lambda) = \text{det}\big(\lambda I - A(\lambda_i(L)) \big) = \lambda^{2}-(1+\zeta)(1-\eta \lambda_i(L)) \lambda+\zeta(1-\eta \lambda_i(L)),
\end{align*}
with discriminant
\begin{align*}
\Delta_i = (1-\eta \lambda_i(L)) \left[(1+\zeta)^2(1-\eta \lambda_i(L)) - 4\zeta \right].
\end{align*}
We split the rest of the proof into three cases.

\paragraph{Case 1: $\lambda_i(L) = \lambda_{N-1}(L)$.} Since $\Delta_i = 0$ in this case, $A(\lambda_{i}(L))$ has two equal eigenvalues given by
\begin{align*}
\lambda_{i,1} = \lambda_{i,2} =1 - \frac{1}{\sqrt{\kappa(L)}}.
\end{align*}
The Jordan normal form of $A(\lambda_i(L))$ is
\begin{align*}
A(\lambda_{i}) &= \frac{1}{\kappa(L)}\begin{bmatrix}
\sqrt{\kappa(L)} \frac{\sqrt{\kappa(L)} - 1}{\sqrt{\kappa(L)} + 1}& \kappa(L)\\
-1 & 0\\
\end{bmatrix} \begin{bmatrix}
1 - \frac{1}{\sqrt{\kappa(L)}} & 1\\
0 & 1 - \frac{1}{\sqrt{\kappa(L)}}\\
\end{bmatrix} \begin{bmatrix}
0 & -\kappa(L)\\
1 & \sqrt{\kappa(L)} \frac{\sqrt{\kappa(L)} - 1}{\sqrt{\kappa(L)} + 1}\\
\end{bmatrix},
\end{align*}
thus for any integer $k \geq 0$, we have
\begin{align*}
A(\lambda_{i}(L))^k & = \begin{bmatrix}
\left(1+\frac{k}{\sqrt{\kappa(L)}+1}\right)\left(1 - \frac{1}{\sqrt{\kappa(L)}}\right)^k & k \left(\frac{\sqrt{\kappa(L)} - 1}{\sqrt{\kappa(L)} + 1}\right)^2 \left(1 - \frac{1}{\sqrt{\kappa(L)}}\right)^{k-1}\\
-\frac{k}{\kappa(L)} \left(1 - \frac{1}{\sqrt{\kappa(L)}}\right)^{k-1}& \left(1-\frac{k}{\sqrt{\kappa(L)}+1}\right)\left(1 - \frac{1}{\sqrt{\kappa(L)}}\right)^k
\end{bmatrix}.
\end{align*}
Consequently,
\begin{align*}
\label{eq:case 1(a)}
\abs{[A(\lambda_i(L))^k]_{1,1}} \leq \left(1+\frac{k}{\sqrt{\kappa(L)}+1}\right)\left(1 - \frac{1}{\sqrt{\kappa(L)}}\right)^k.
\end{align*}
Moreover, for any integers $j \geq k \geq 0$,
\begin{align*}
\abs{[A(\lambda_i(L))^k]_{1,2} \cdot [A(\lambda_i(L))^j]_{2,1}} \overset{(a)}{\leq} \frac{k^2}{\kappa(L)} \left(1 - \frac{1}{\sqrt{\kappa(L)}}\right)^{j+k} \overset{(b)}{\leq} \left(1+\frac{k}{\sqrt{\kappa(L)}+1}\right)\left(1 - \frac{1}{\sqrt{\kappa(L)}}\right)^k.
\end{align*}
where step (a) uses the inequality 
\begin{align*}
\frac{\sqrt{\kappa(L)} - 1}{\sqrt{\kappa(L)} + 1} \leq \frac{\sqrt{\kappa(L)} - 1}{\sqrt{\kappa(L)}} = 1 - \frac{1}{\sqrt{\kappa(L)}},
\end{align*}
and step (b) is due to the bound
\begin{align*}
\frac{k^2}{\kappa(L)} \left(1 - \frac{1}{\sqrt{\kappa(L)}}\right)^{j} = \frac{k}{\sqrt{\kappa(L)}} \cdot \frac{k}{\sqrt{\kappa(L)}} e^{- \frac{j}{\sqrt{\kappa(L)}}} \leq \left(1+\frac{k}{\sqrt{\kappa(L)}+1}\right).
\end{align*}
\paragraph{Case 2: $\lambda_i(L) = \lambda_1(L)$.} Since $\Delta_i = 0$ in this case, $A(\lambda_i(L))$ has equal eigenvalues given by
\begin{align*}
\lambda_{i,1} = \lambda_{i,2} = 0.
\end{align*}
The Jordan normal form of $A(\lambda_i(L))$ is
\begin{align*}
A(\lambda_i(L)) = \begin{bmatrix}
\zeta & -1\\
1 & 0\\
\end{bmatrix} \begin{bmatrix}
0 & 1\\
0 & 0\\
\end{bmatrix} \begin{bmatrix}
0 & 1\\
-1 & \zeta\\
\end{bmatrix};
\end{align*}
therefore
\begin{align*}
A(\lambda_i(L)) = \begin{bmatrix}
-\zeta & \zeta^2\\
-1 & \zeta\\
\end{bmatrix}, ~\text{ and }~ A(\lambda_i(L))^k = \begin{bmatrix}
0 & 0\\
0 & 0\\
\end{bmatrix} \text{ for any integer } k \geq 2.
\end{align*}
Some straightforward calculations then yield that for any integers $j \geq k \geq 0$:
\begin{align*}
\abs{[A(\lambda_i(L))^k]_{1,1}} \leq \left(1+\frac{k}{\sqrt{\kappa(L)}+1}\right)\left(1 - \frac{1}{\sqrt{\kappa(L)}}\right)^k,
\end{align*}
and
\begin{align*}
\abs{[A(\lambda_i(L))^k]_{1,2} \cdot [A(\lambda_i(L))^j]_{2,1}} \leq \left(1+\frac{k}{\sqrt{\kappa(L)}+1}\right)\left(1 - \frac{1}{\sqrt{\kappa(L)}}\right)^k.
\end{align*}

\paragraph{Case 3: $\lambda_{N-1}(L) < \lambda_i(L) < \lambda_1(L)$.} Since
\begin{align*}
\Delta_i=\frac{4(1-\frac{\lambda_i(L)}{\lambda_1})\left(1- \frac{\lambda_i(L)}{\lambda_{N-1}}\right)}{\left(\sqrt{\kappa(L)}+1\right)^2} < 0
\end{align*}
in this case, $A(\lambda_{i}(L))$ has two complex eigenvalues given by
\begin{align*}
\lambda_{i,1} = \frac{(1+\zeta)(1-\eta \lambda_i(L))}{2}+\frac{\sqrt{-\Delta_i}}{2}i =\sqrt{\zeta\left(1-\eta \lambda_i(L) \right)}e^{i\phi_i},
\end{align*}
and
\begin{equation*}
\begin{split}
\lambda_{i,2} = \frac{(1+\zeta)(1-\eta \lambda_i(L))}{2}-\frac{\sqrt{-\Delta_i}}{2}i =\sqrt{\zeta\left(1-\eta \lambda_i(L) \right)}e^{-i\phi_i},
\end{split}
\end{equation*}
where $\phi_i$ satisfies
\begin{align*}
\cos{\phi_i} = \frac{(1+\zeta)(1-\eta \lambda_i(L))}{2\sqrt{\zeta\left(1-\eta \lambda_i(L) \right)}} ~\text{ and }~ \sin{\phi_i} = \frac{\sqrt{-\Delta_i}}{2\sqrt{\zeta\left(1-\eta \lambda_i(L) \right)}}.
\end{align*}
Thus, $A(\lambda_i(L)) = V(\lambda_i(L)) C(\lambda_i(L)) V(\lambda_i(L))^{-1}$, where
\begin{align*}
C(\lambda_i(L)) = \sqrt{\zeta\left(1-\eta \lambda_i(L) \right)}\begin{bmatrix}
\cos{\phi_i} & -\sin{\phi_i}\\
\sin{\phi_i} & \cos{\phi_i}\\
\end{bmatrix},
\end{align*}

\begin{align*}
V(\lambda_i(L)) = \begin{bmatrix}
\zeta-\frac{(1+\zeta)(1-\eta \lambda_i(L))}{2} & \frac{\sqrt{-\Delta_i}}{2}\\
\eta \lambda_i(L) & 0\\
\end{bmatrix},
\end{align*}
and
\begin{align*}
V(\lambda_i(L))^{-1} = \frac{-2}{\eta \lambda_i(L)\sqrt{-\Delta_i}}\begin{bmatrix}
0 & -\frac{\sqrt{-\Delta_i}}{2}\\
-\eta \lambda_i(L) & \zeta-\frac{(1+\zeta)(1-\eta \lambda_i(L))}{2}\\
\end{bmatrix}.
\end{align*}
So we have for any integer $k \geq 0$:
\begin{align*}
\begin{split}
&\quad A(\lambda_i(L))^k = \\
&\left(\sqrt{\zeta\left(1-\eta \lambda_i(L) \right)}\right)^k \begin{bmatrix}
  \cos{(k\phi_i)} + \frac{\frac{(1+\zeta)(1-\eta \lambda_i(L))}{2} - \zeta}{\frac{\sqrt{-\Delta_i}}{2}} \sin{(k\phi_i)}& \frac{\zeta^2}{\frac{\sqrt{-\Delta_i}}{2}}\sin{(k\phi_i)}\\
\frac{-\eta \lambda_i(L)}{\frac{\sqrt{-\Delta_i}}{2}} \sin{(k\phi_i)} & \cos{(k\phi_i)} - \frac{\frac{(1+\zeta)(1-\eta \lambda_i(L))}{2} - \zeta}{\frac{\sqrt{-\Delta_i}}{2}} \sin{(k\phi_i)}\\
\end{bmatrix}.
\end{split}
\end{align*}
Since 
\begin{align*}
\abs{\left(\sqrt{\zeta\left(1-\eta \lambda_i(L) \right)}\right)^k} &\leq \left(1 - \frac{1}{\sqrt{\kappa(L)}} \right)^{k},
\end{align*}
and
\begin{align*}
\abs{\cos{(k\phi_i)} + \frac{\frac{(1+\zeta)(1-\eta \lambda_i(L))}{2} - \zeta}{\frac{\sqrt{-\Delta_i}}{2}} \sin{(k\phi_i)}} &\leq 1 + \abs{\frac{\frac{(1+\zeta)(1-\eta \lambda_i(L))}{2} - \zeta}{\sqrt{\zeta\left(1-\eta \lambda_i(L) \right)}}} \abs{\frac{\sin{(k\phi_i)}}{\sin{(\phi_i)}}},
\end{align*}
we have that for any integer $k \geq 1$:
\begin{align*}
\abs{[A(\lambda_i(L))^k]_{1,1}} &\leq \left(1 - \frac{1}{\sqrt{\kappa(L)}} \right)^{k-1} \left(1 + \abs{\frac{\frac{(1+\zeta)(1-\eta \lambda_i(L))}{2} - \zeta}{\sqrt{\zeta\left(1-\eta \lambda_i(L) \right)}}} \abs{\frac{\sin{(k\phi_i)}}{\sin{(\phi_i)}}}\right).
\end{align*}
Since for any integer $k \geq 1$:
\begin{align*}
\sin{(k\phi_i)} = \left(\sum^{k-1}_{j=0} \cos{(j \cdot \phi_i)} \cdot \cos{(\phi_i)}^{k-1-j} \right) \cdot \sin{(\phi_i)},
\end{align*}
then we have
\begin{align*}
 \abs{\frac{\sin{(k\phi_i)}}{\sin{(\phi_i)}}} &= \abs{\sum^{k-1}_{j=0} \cos{(j \cdot \phi_i)} \cdot \cos{(\phi_i)}^{k-1-j}} \leq k.
\end{align*}
Now define the function
\begin{align*}
\begin{split}
f(x) :&= \left[\left(\frac{(1+\zeta)(1-\eta x)}{2} - \zeta\right)\left(\sqrt{\kappa(L)}+1\right)\right]^2-\zeta\left(1-\eta x \right)\\
&= \frac{\kappa(L)}{\lambda_1(L)^2}x^2 + \left[\frac{\sqrt{\kappa(L)} - 1}{\left(\sqrt{\kappa(L)} + 1\right)\lambda_1(L)} - \frac{2\sqrt{\kappa(L)}}{\lambda_1(L)}\right]x + 1-\frac{\sqrt{\kappa(L)} - 1}{\sqrt{\kappa(L)} + 1}.
\end{split}
\end{align*}
Since $f(\lambda_1(L)) = 0$, $f(\lambda_{N-1}(L)) = 0$, and $\frac{\kappa(L)}{\lambda_1(L)^2} > 0$, we have $f(\lambda_i(L)) < 0$ when $\lambda_{N-1}(L) < \lambda_i(L) < \lambda_1(L)$. 
Consequently, for any integers $j \geq k \geq 0$, we have
\begin{align*}
\abs{[A(\lambda_i(L))^k]_{1,1}} \leq \left(1+\frac{k}{\sqrt{\kappa(L)}+1}\right)\left(1 - \frac{1}{\sqrt{\kappa(L)}}\right)^k,
\end{align*}
and
\begin{align*}
\begin{split}
&\abs{[A(\lambda_i(L))^k]_{1,2} \cdot [A(\lambda_i(L))^j]_{2,1}} \\
\leq & \abs{\left(\sqrt{\zeta\left(1-\eta \lambda_i(L) \right)}\right)^{j-1}\eta \lambda_i(L)\zeta \frac{\sin{(k\phi_i)}}{\sin{(\phi_i)}} \frac{\sin{(j\phi_i)}}{\sin{(\phi_i)}}} \left(1 - \frac{1}{\sqrt{\kappa(L)}} \right)^{k}\\
\leq &  \left(1+\frac{k}{\sqrt{\kappa(L)}+1}\right)\left(1 - \frac{1}{\sqrt{\kappa(L)}}\right)^k.
\end{split}
\end{align*}

Finally, since we have shown that bounds \eqref{eq:lemmq2 (a)} and \eqref{eq:lemmq2 (b)} hold for all $i \in \mathcal{N}$ and integers $j \geq k \geq 0$, the proof is complete.
\end{proof}

\begin{lemma}
\label{lemma:k^* upper bound}
Let $k^*$ be defined according to Definition \ref{definition:k^*}. Then there exists an absolute constant $C \geq 1$ such that $k^* \leq C \cdot \sqrt{\kappa(L)}$.
\end{lemma}
\begin{proof}
First of all,  we show that $k^*$ is well-defined. When $\kappa(L) = 1$, we have $k^* = 1$ because $\left(1 - \frac{1}{\sqrt{\kappa(L)}}\right) = 0$ and thus Eq.~\eqref{eq:k^* definition} holds for all $k \geq 1$. When $\kappa(L) > 1$, Eq.~\eqref{eq:k^* definition} is equivalent to 
\begin{align}
\label{eq:k^* equivalent form}
\left(1+\frac{k}{\sqrt{\kappa(L)}+1}\right) \leq \left(\frac{1 - \frac{1}{2\sqrt{\kappa(L)}}}{1 - \frac{1}{\sqrt{\kappa(L)}}}\right)^{k} &= \left(1 + \frac{1}{2\left(\sqrt{\kappa(L)} - 1\right)}\right)^k.
\end{align}
As the left-hand side is linear in $k$ while the right-hand side is exponential in $k$, for any $\kappa(L) > 1$, there exists a finite smallest positive integer $k^*$ such that Eq.~\eqref{eq:k^* equivalent form} holds for all integer $k \geq k^*$. Next we show that $k^* \leq C \cdot \kappa(L)$ for some absolute constant $C$. Since
\begin{align*}
1 + \frac{1}{2\left(\sqrt{\kappa(L)} - 1\right)} \geq e^{\frac{1}{2\sqrt{\kappa(L)}}}
\end{align*}
holds for all $\kappa(L) > 1$, then for any integer $k \geq 0$: \begin{align*}
\left(1 + \frac{1}{2\left(\sqrt{\kappa(L)} - 1\right)}\right)^k \geq e^{\frac{k}{2\sqrt{\kappa(L)}}}.
\end{align*}
Let $k = 3\sqrt{\kappa(L)}$, then we have
\begin{align*}
\left(1+\frac{k}{\sqrt{\kappa(L)}+1}\right) < 4 < e^{1.5} = e^{\frac{k}{2\sqrt{\kappa(L)}}} \leq \left(1 + \frac{1}{2\left(\sqrt{\kappa(L)} - 1\right)}\right)^k =\left(\frac{1 - \frac{1}{2\sqrt{\kappa(L)}}}{1 - \frac{1}{\sqrt{\kappa(L)}}}\right)^{k}. 
\end{align*}
Thus, $k^* \leq 3\sqrt{\kappa(L)}$ for all $\kappa(L) > 1$. Since $k^* = 1 \leq \sqrt{\kappa(L)}$ when $\kappa(L) = 1$, there must exist an absolute constant $C \geq 1$ such that $k^* \leq C \cdot \sqrt{\kappa(L)}$.
\end{proof}

\begin{lemma}
\label{lemma:infinite sum upper bound}
Let $k^*$ be defined according to Definition \ref{definition:k^*}. Then we have
\begin{align*}
\sum^{+\infty}_{k=0} \left[\left(1+\frac{k}{\sqrt{\kappa(L)}+1}\right)\left(1 - \frac{1}{\sqrt{\kappa(L)}}\right)^k\right]^2 \leq k^* + \sqrt{\kappa(L)}.
\end{align*}
\end{lemma}
\begin{proof}
Using the fact that
\begin{align*}
\left(1+\frac{k}{\sqrt{\kappa(L)}+1}\right)\left(1 - \frac{1}{\sqrt{\kappa(L)}}\right)^k \leq e^{\frac{k}{\sqrt{\kappa(L)}+1}} \cdot e^{-\frac{k}{\sqrt{\kappa(L)}}} \leq 1 \quad \text{for all } 0 \leq k \leq k^*-1,
\end{align*}
and Eq.~\eqref{eq:k^* definition}, we have
\begin{align*}
\begin{split}
\sum^{+\infty}_{k=0} \left[\left(1+\frac{k}{\sqrt{\kappa(L)}+1}\right)\left(1 - \frac{1}{\sqrt{\kappa(L)}}\right)^k\right]^2 &\leq k^* + \sum^{\infty}_{k=k^*} \left[\left(1+\frac{k}{\sqrt{\kappa(L)}+1}\right)\left(1 - \frac{1}{\sqrt{\kappa(L)}}\right)^k\right]^2\\
&\leq k^* + \sum^{\infty}_{k=k^*} \left(1 - \frac{1}{2\sqrt{\kappa(L)}}\right)^{2k}\\
&\leq k^* + \sqrt{\kappa(L)}.\\
\end{split}
\end{align*}
\end{proof}

\begin{lemma}
\label{lemma:upper bound inverse of a matrix}
Suppose $A$ and $B$ are invertible real square matrices such that $\norm{A^{-1} \left(B - A \right)}_2  < 1$. Then we have
\begin{align}
    \label{eq:upper bound inverse of a matrix}
\norm{A^{-1} - B^{-1}}_2 \leq \frac{\norm{A^{-1}}^2_2 \norm{A - B}_2}{1 - \norm{A^{-1} \left(B - A \right)}_2}.
\end{align}
\end{lemma}
\begin{proof}
Since $A^{-1} - B^{-1} = A^{-1} \left(B - A \right)B^{-1}$, we have 
\begin{align}
\label{eq:upper bound inverse of a matrix (1)}
\norm{A^{-1} - B^{-1}}_2 = \norm{A^{-1} \left(B - A \right)B^{-1}}_2 \leq \norm{A^{-1}}_2 \norm{A-B}_2 \cdot \norm{B^{-1}}_2. 
\end{align}
Moreover, since $B^{-1} = A^{-1} - A^{-1} \left(B - A \right)B^{-1}$, we obtain 
\begin{align*}
\norm{B^{-1}}_2 &= \norm{A^{-1} - A^{-1} \left(B - A \right)B^{-1}}_2 \leq \norm{A^{-1}}_2 + \norm{A^{-1} \left(B - A \right)}_2 \cdot \norm{B^{-1}}_2, 
\end{align*}
which is equivalent to 
\begin{align}
\label{eq:upper bound inverse of a matrix (2)}
\left(1 - \norm{A^{-1} \left(B - A \right)}_2 \right) \norm{B^{-1}}_2 \leq \norm{A^{-1}}_2.
\end{align}
Since $\norm{A^{-1} \left(B - A \right)}_2 < 1$, it follows from Eq.~\eqref{eq:upper bound inverse of a matrix (2)} that
\begin{align}
\label{eq:upper bound inverse of a matrix (3)}
\norm{B^{-1}}_2 \leq \frac{\norm{A^{-1}}_2}{1 - \norm{A^{-1} \left(B - A \right)}_2}.
\end{align}
Finally, substituting Eq.~\eqref{eq:upper bound inverse of a matrix (3)} into Eq.~\eqref{eq:upper bound inverse of a matrix (1)}, we obtain the desired result.
\end{proof}

%% file: conclusion.tex
\section{Discussion}
We considered an online, stochastic distributed averaging problem in which noisy data becomes available sequentially to agents over a network. We developed a stochastic dual accelerated method for this problem using constant step-size and Polyak--Ruppert averaging. We showed that this simple algorithm attains an accelerated deterministic error depending optimally on the connectivity parameter of the network, and also that it has an order-optimal stochastic error. This improves on the guarantees of state-of-the-art distributed stochastic optimization algorithms when specialized to this setting. Our proofs explicitly studied the evolution of several relevant linear systems, and may be of independent interest. 

Let us conclude by mentioning a few future directions. One drawback of our setting is that the communication network is static. In many applications, the underlying connectivity structure of the network may vary with time, so a future direction is to extend our approach to this more challenging setting. Another drawback, from the perspective of distributed optimization, is that our setting only considers the special class of quadratic objective functions. Thus, a natural next step is to generalize our dual method to strongly convex and smooth local objective functions. This would allow us to make progress on the design of optimal dual-based algorithms under this setting, which is known to be an open problem in the literature \citep{gorbunov2020recent}.

%% file: appendix-DSG.tex
\section{Distributed Stochastic Gradient Method}
\label{sec:distributed stochastic gradient}



\begin{algorithm}
\caption{Distributed Stochastic Gradient Method (\texttt{DSG})}\label{alg:Distributed Stochastic Gradient Method}
\begin{algorithmic}[1]
\State \textbf{Input:} number of iterations $T>0$, weight matrix $W \in R^{N \times N}$ and step-size sequence $\{\eta_t\}^{T-1}_{t=0}$.
\State \textbf{Initialization:} each agent $i \in \mathcal{N}$ initializes its local estimate $\widehat{\theta}_{i,0} = \mathbf{0} \in \mathbb{R}^n$.
\For{$t=0,1,\ldots,T-1$}
\For{each agent $i \in \mathcal{N}$}
\State observes a local random vector $R_{i,t}$ and executes the local computation:
\begin{equation}
\label{eq:Online Distributed Linear Itearative Algorithm Diminishing Step-Sizes Update Approximation}
\begin{split}
\tilde{\theta}_{i,t} = \widehat{\theta}_{i,t} + \eta_t \left(R_{i,t} - \widehat{\theta}_{i,t} \right).
\end{split} 
\end{equation}
\State exchanges $\tilde{\theta}_{i,t}$ with agent $j \in \mathcal{N}_i$ and executes the local update:
\begin{equation}
\label{eq:Online Distributed Linear Itearative Algorithm Diminishing Step-Sizes Update Consensus}
\begin{split}
\widehat{\theta}_{i,t+1} = \sum_{j \in \mathcal{N}_i \cup \{i\}} W_{i,j} \tilde{\theta}_{j,t}.
\end{split} 
\end{equation}
\EndFor
\EndFor
\State \textbf{Output:} $\widehat{\theta}_{i,T}$ for all $i \in \mathcal{N}$.
\end{algorithmic}
\end{algorithm}
In this section, we present the details of \texttt{DSG}, the baseline of our first set of experiments in Section \ref{sec:numerical results}, and provide its theoretical guarantee in Theorem \ref{theorem:finite-time upper bound of Consensus-Based Stochastic Gradient Method}. The method is formally stated in Algorithm \ref{alg:Distributed Stochastic Gradient Method}, which can be explained as follows. Each node $i \in \mathcal{N}$ maintains its own estimate $\widehat{\theta}_i$ of the target vector $\bar{\mu}$. At every iteration $t$, agent $i$ first computes the intermediate value $\tilde{\theta}_{i,t}$ by moving along its own negative stochastic gradient direction $R_{i,t} - \widehat{\theta}_{i,t}$, pushing the estimate toward $\mu_i$. Then, agent $i$ updates the estimate $\widehat{\theta}_{i,t}$ by forming a weighted average of its own intermediate value $\tilde{\theta}_{i,t}$ and the intermediate values $\{\tilde{\theta}_{j,t}\}_{j \in \mathcal{N}_i}$ received from its neighbors, with the goal of seeking consensus on their estimates.

We make the following standard assumptions \citep{xiao2004fast,xiao2007distributed} about the weight matrix $W = [W_{i,j}] \in \mathbb{R}^{N \times N}$:
\begin{itemize}
    \item[1.] $W$ is symmetric: $W = W^\top$.
    \item[2.] The all-ones vector $e_N \in \mathbb{R}^N$ is an eigenvector of $W$ associated with eigenvalue one: $We_N = e_N$.
    \item[3.] $W$ is defined on the edges of the network: $W_{i,j} \not = 0$ only if $\{i,j\} \in \mathcal{E}$ or $i = j$.
    \item[4.] The spectral radius of $W - \frac{1}{N}e_N e^\top_N$ is strictly less than one: $\rho\left( W - \frac{1}{N}e_N e^\top_N \right) < 1$.
\end{itemize}
These conditions ensures that (i) one is a simple eigenvalue of $W$, and all other eigenvalues are strictly less than one in absolute value; (ii) $W^t \rightarrow \frac{1}{N}e_N e^\top_N$ as $t \rightarrow \infty$. We denote by $\lambda_{(i)}(W)$ the $i$-th largest eigenvalue of $W$ in magnitude, and thus $1 = \lambda_{(1)}(W) > \abs{\lambda_{(2)}(W)} \geq \cdots \geq \abs{\lambda_{(N)}(W)}$. So the condition number $\kappa(W)$ of $W$ is defined as
\begin{align*}
\kappa(W):= \frac{1}{1 - \abs{\lambda_{(2)}(W)}}.
\end{align*}
It is worth noting that given a weight matrix $W$ satisfying the above conditions, we can easily construct a gossip matrix $L$ by $L := I_N - W$, which automatically satisfies the assumptions on $L$ in Section \ref{sec:stochastic dual accelerated method}. This construction of $L$ from $W$ has also been used in our numerical experiments; see Appendix \ref{sec:implementation details} for details. 

We now present the finite-time performance upper bound of \texttt{DSG}.
\begin{theorem}
\label{theorem:finite-time upper bound of Consensus-Based Stochastic Gradient Method}
Consider running \texttt{DSG} with the following parameters: 
\begin{align*}
T > 0 ~\text{ and }~  \eta_t = \frac{1}{t+1} \text{ for all } t=0,1,\ldots,T-1.
\end{align*} 
Let $\{\widehat{\theta}_{i,T}\}_{i \in \mathcal{N}}$ be generated by \texttt{DSG}. Then we have
\begin{align*}
\mathbb{E}\left[\sum^N_{i=1}\norm{\widehat{\theta}_{i,T} - \bar{\mu}}^2_2\right] \leq \frac{2\kappa(W)^2\sum^N_{i=1}\norm{\mu_i}^2_2}{T^2} + \frac{2}{T^2} \sum^{N-1}_{i=1}\frac{\sum^n_{s=1}\sigma^2_{(i),s}}{1 - \abs{\lambda_{(i+1)}\left(W\right)}^2} + \frac{2\sum^N_{i=1}\sum^n_{s=1} \sigma^2_{i,s}}{NT},
\end{align*}
where $\sigma^2_{(i), s}$ is the $i$-th largest variance among $\{\sigma^2_{1,s}, \ldots, \sigma^2_{N,s}\}$ for all $i = 1, \ldots, N$ and $s=1,\ldots,n$.
\end{theorem}
\begin{proof} 
Fix an arbitrary $s \in \{1,\ldots,n\}$.
Combining Eq. \eqref{eq:Online Distributed Linear Itearative Algorithm Diminishing Step-Sizes Update Approximation} and \eqref{eq:Online Distributed Linear Itearative Algorithm Diminishing Step-Sizes Update Consensus}, we have the following recursion for all integer $t \geq 0$:
\begin{equation}
\label{eq:Online Distributed Average Consensus Update in matrix form}
\begin{split}
\widetilde{\Theta}_{t+1,s} = W \Big[\widetilde{\Theta}_{t,s} + \eta_t \left(\widetilde{R}_{t,s} - \widetilde{\Theta}_{t,s} \right)\Big]\\
\end{split} 
\end{equation}
where
\begin{equation*}
\begin{split}
\widetilde{\Theta}_{t,s} = \begin{bmatrix}
\theta_{1,t,s} \cdots \theta_{N,t,s} \end{bmatrix}^\top \text{ and } \widetilde{R}_{t,s} = \begin{bmatrix}
 R_{1,t,s} \cdots R_{N,t,s} 
\end{bmatrix}^\top.
\end{split} 
\end{equation*}
Note that, we use $\theta_{i,t,s}$ and $R_{i,t,s}$ to denote the $s$-th component of the vectors $\theta_{i,t}$ and $R_{i,t}$, respectively. Let
\begin{equation*}
\begin{split}
\bar{\Theta}_{t,s} :=\frac{1}{N}e_N^\top \widetilde{\Theta}_{t,s} \quad \text{ and } \quad \bar{R}_{t,s} := \frac{1}{N} e_N^\top \widetilde{R}_{t,s},
\end{split} 
\end{equation*}
then by Eq. \eqref{eq:Online Distributed Average Consensus Update in matrix form} we have the the following recursion for all integer $t \geq 0$:
\begin{equation}
\label{eq:Online Distributed Average Consensus Update average estimate}
\begin{split}
\bar{\Theta}_{t+1,s} = \bar{\Theta}_{t,s} + \eta_t \left(\bar{R}_{t,s} - \bar{\Theta}_{t,s} \right).
\end{split} 
\end{equation}
Substituting $\eta_t = \frac{1}{t+1}$ into Eq. \eqref{eq:Online Distributed Average Consensus Update average estimate} and unrolling the recursion, we obtain that
\begin{equation*} 
\begin{split}
\bar{\Theta}_{T,s} = \frac{1}{T} \sum^{T-1}_{t=0}\bar{R}_{t,s}.
\end{split} 
\end{equation*}
Since $\mathbb{E}\left[\bar{\Theta}_{T,s}\right] = \bar{\mu}_s$ and $\text{var}\left(\bar{\Theta}_{T,s}\right) = \frac{1}{N^2T} \sum^{N}_{i=1}\sigma^2_{i,s}$, we have
\begin{equation}
\label{eq:average error}
\begin{split}
\mathbb{E}\left[\norm{ \bar{\Theta}_{T,s} e_N - \bar{\mu}_s e_N }^2_{2} \right] = N\mathbb{E}\left[\left(\bar{\Theta}_{T,s} - \bar{\mu}_s\right)^2_2\right] = \frac{1}{NT}\sum^{N}_{i=1}\sigma^2_{i,s}.
\end{split} 
\end{equation}
In addition, if we subtract both sides of Eq. \eqref{eq:Online Distributed Average Consensus Update in matrix form} by $\bar{\Theta}_{T,s} e_N$ and unroll the recursion, then we have
\begin{equation*}
\begin{split}
{\widetilde{\Theta}}_{T,s} - \bar{\Theta}_{T,s} e_N = \sum^{T-1}_{t=0} \eta_{t} \left[\prod^{T-1}_{i=t+1} \left(1-\eta_{i}\right)\right] \left(W - \frac{1}{N} e_N e^\top_N \right)^{T-t}\widetilde{R}_{t,s}.
\end{split} 
\end{equation*}
Since $\eta_t = \frac{1}{t+1}$ for all $t \geq 0$,  we have
\begin{equation*}
\begin{split}
\eta_{t} \left[\prod^{T-1}_{i=t+1} \left(1-\eta_{i}\right)\right] = \frac{1}{T} \quad \text{for all } t = 0,\ldots,T-1,
\end{split} 
\end{equation*}
from which it follows that
\begin{equation*}
\begin{split}
{\widetilde{\Theta}}_{T,s} - \bar{\Theta}_{T,s} e_N=\frac{1}{T}\sum^{T-1}_{t=0} \left(W - \frac{1}{N} e_N e^\top_N \right)^{T-t} \widetilde{R}_{t,s}.\\
\end{split} 
\end{equation*}
Hence,
\begin{equation}
\label{eq:consensus error}
\begin{split}
\mathbb{E}\left[\norm{ \widetilde{\Theta}_{T,s} - \bar{\Theta}_{T,s} e_N }^2_{2}\right] &= \frac{1}{T^2} \mathbb{E}\left[\norm{\sum^{T-1}_{t=0} \left(W - \frac{1}{N} e_N e^\top_N \right)^{T-t} \widetilde{R}_{t,s}}^2_{2}\right] \\
&= \frac{1}{T^2} \sum^{T}_{k=1}\text{tr}\left(\left(W - \frac{1}{N} e_N e^\top_N \right)^{2k} \widetilde{\Sigma}_s\right) \\
&\quad+ \frac{1}{T^2} \sum^{T}_{i=1}\sum^{T}_{j=1}\text{tr}\left(\left(W - \frac{1}{N} e_N e^\top_N \right)^{i+j} \widetilde{\mu}_s \widetilde{\mu}_s^\top \right)\\
&\overset{(a)}{\leq} \frac{1}{T^2}\sum^{T}_{k=1} \sum^{N-1}_{i=1} \abs{\lambda_{(i+1)}\left(W\right)}^{2k} \cdot \sigma^2_{(i),s} + \frac{\norm{\widetilde{\mu}_s}^2_2}{T^2} \sum^{T}_{j=1}\sum^{T}_{k=1} \abs{\lambda_{(2)}(W)}^{j+k}\\
&\overset{(b)}{=} \frac{1}{T^2} \sum^{N-1}_{i=1} \sigma^2_{(i),s} \sum^{T}_{k=1} \abs{\lambda_{(i+1)}\left(W\right)}^{2k} + \frac{\norm{\widetilde{\mu}_s}^2_2}{T^2} \left[\sum^{T}_{j=1} \abs{\lambda_{(2)}(W)}^{j}\right]^2\\
&\overset{(c)}{\leq} \frac{1}{T^2} \sum^{N-1}_{i=1}\frac{\sigma^2_{(i),s}}{1 - \abs{\lambda_{(i+1)}\left(W\right)}^2} + \frac{\kappa(W)^2\norm{\widetilde{\mu}_s}^2_2}{T^2},\\
\end{split} 
\end{equation}
where $\widetilde{\mu}_s := \begin{bmatrix} 
\mu_{1, s} \cdots \mu_{N, s}
\end{bmatrix}^\top$ and $\widetilde{\Sigma}_s := \begin{bmatrix}
\sigma^2_{1,s} & &\\
 & \ddots & \\
 & & \sigma^2_{N,s}
\end{bmatrix}$. Here, the step (a) is because the eigenvalues of $W - \frac{1}{N} e_N e^\top_N$ are  $\lambda_{(2)}\left(W\right)$, $\ldots$, $\lambda_{(N)}\left(W\right)$ and $0$, and the Von Neumann's trace inequality implies that $$\text{tr}\left(\left(W - \frac{1}{N} e_N e^\top_N \right)^{2k} \widetilde{\Sigma}_s\right) \leq \sum^{N-1}_{i=1} \abs{\lambda_{(i+1)}\left(W\right)}^{2k} \cdot \sigma^2_{(i),s} \quad \text{for all } k = 1,\ldots, T,$$ 
and furthermore we use the following property of the symmetric matrix $W - \frac{1}{N} e_N e^\top_N$:
$$\text{tr}\left(\left(W - \frac{1}{N} e_N e^\top_N \right)^{i+j} \widetilde{\mu}_s \widetilde{\mu}_s^\top \right) = \widetilde{\mu}_s^\top  \left(W - \frac{1}{N} e_N e^\top_N \right)^{i+j} \widetilde{\mu}_s \leq  \abs{\lambda_{(2)}(W)}^{j+k} \norm{\widetilde{\mu}_s}^2_2.$$
The step (b) uses the identity that
$$\left[\sum^{T}_{j=1} \abs{\lambda_{(2)}(W)}^{j}\right]^2 = \sum^{T}_{j=1}\sum^{T}_{k=1} \abs{\lambda_{(2)}(W)}^{j+k}.$$
The step (c) is because
$$\sum^{T}_{k=1} \abs{\lambda_{(i+1)}\left(W\right)}^{2k} \leq \sum^{\infty}_{k=0} \abs{\lambda_{(i+1)}\left(W\right)}^{2k} = \frac{1}{1 - \abs{\lambda_{(i+1)}\left(W\right)}^2}$$
and
$$\left[\sum^{T}_{j=1} \abs{\lambda_{(2)}(W)}^{j}\right]^2 \leq \left[\sum^{\infty}_{j=0} \abs{\lambda_{(2)}(W)}^{j}\right]^2 \leq \left[\frac{1}{1 - \abs{\lambda_{(2)}(W)}}\right]^2 = \kappa(W)^2.$$
Finally, Combining Eq. \eqref{eq:average error} and \eqref{eq:consensus error}, we obtain that 
\begin{equation*}
\begin{split}
\mathbb{E}\left[\sum^N_{i=1}\norm{\theta_{i,T} - \bar{\mu}}^2_{2} \right] &= \sum^{n}_{s=1}\mathbb{E}\left[\norm{\widetilde{\Theta}_{T,s} - \bar{\mu}_s e_N}^2_{2} \right]\\
&= \sum^{n}_{s=1}\mathbb{E}\left[\norm{\widetilde{\Theta}_{T,s} - \bar{\Theta}_{T,s} e_N + \bar{\Theta}_{T,s} e_N - \bar{\mu}_s e_N}^2_{2} \right]\\
&\leq 2\sum^{n}_{s=1}\mathbb{E}\left[\norm{\widetilde{\Theta}_{T,s} - \bar{\Theta}_{T,s} e_N }^2_{2} \right] +  2\sum^{n}_{s=1}\mathbb{E}\left[\norm{ \bar{\Theta}_{T,s} e_N - \bar{\mu}_s e_N }^2_{2} \right]\\
&\leq \frac{2\kappa(W)^2\sum^N_{i=1}\norm{\mu_i}^2_2}{T^2} + \frac{2}{T^2} \sum^{N-1}_{i=1}\frac{\sum^n_{s=1}\sigma^2_{(i),s}}{1 - \abs{\lambda_{(i+1)}\left(W\right)}^2} + \frac{2\sum^N_{i=1}\sum^n_{s=1} \sigma^2_{i,s}}{NT}.
\end{split} 
\end{equation*}
\end{proof}

%% file: appendix-implementation.tex
\section{Implementation Details}
\label{sec:implementation details}
We first introduce the Metropolis-Hastings (MH) weight matrix $W_{MH} = [W_{MH}(i, j)] \in \mathbb{R}^{N \times N}$ associated with a graph $\mathcal{G}$ \citep{xiao2007distributed}:
\begin{equation}
\label{eq:MH weight matrix}
\begin{split}
W_{MH}(i, j) = \begin{cases}
  \frac{1}{\max\{\vert \mathcal{N}_i \vert, \vert \mathcal{N}_j \vert\}+1}, & \text{if } j \in \mathcal{N}_i, \\
  1 - \sum_{j \in \mathcal{N}_i} \frac{1}{\max\{\vert \mathcal{N}_i \vert, \vert \mathcal{N}_j \vert\}+1}, & \text{if } i=j,\\
  0, & \text{otherwise}.
  \end{cases}
\end{split} 
\end{equation}
Note that, $W_{MH}$ satisfies
\begin{equation*}
\begin{split}
W_{MH} = W_{MH}^\top, \quad W_{MH}e_N = e_N, \quad \text{and } \Big \Vert W_{MH} - \frac{1}{N} e_N e_N^\top \Big \Vert_2 < 1.
\end{split} 
\end{equation*}
Therefore, it yields the asymptotic average consensus:
\begin{equation*}
\begin{split}
\lim_{t \rightarrow \infty} W^t_{MH} = \frac{1}{N} e_N e_N^\top. 
\end{split} 
\end{equation*}
We consider $W_{MH}$ as the default weight matrix associated with each graph in Figure \ref{fig:5 networks}, so every implemented algorithm either uses $W_{MH}$ or a transformation of $W_{MH}$.

\subsection{Details of experiments in Section \ref{sec:numerical results convergence behavior}}
Since \texttt{D-MASG} has more parameters to set than \texttt{SDA} and \texttt{DSG}, we first deal with the choice of parameters for \texttt{D-MASG} according to Corollary 18 in \citet{fallah2019robust}. The weight matrix $W_1$ used by \texttt{D-MASG} is a shifted version of $W_{MH}$ defined by $W_1:= \frac{1}{2} I_N + \frac{1}{2} W_{MH}$ (see Remark 9 in \citet{fallah2019robust} for details). Let $\lambda^{W_1}_N$ denote the smallest eigenvalue of $W_1$. Then the scaled condition number is defined as $\tilde{\kappa}:= 2/\lambda^{W_1}_N$. \texttt{D-MASG} consists of $K>0$ stages and the $k$-th stage length is denoted by $t_k$ for $k \in \{1,\ldots,K\}$. We set $t_k =2^k \Big \lceil 7 \sqrt{\tilde{\kappa}} \ln(2) \Big \rceil$ for $k \in \{2,\ldots,K\}$ and then set $t_1 = \sum^K_{k=2} t_k$. Thus, the number of iterations is $T=\sum^K_{k=1}t_k = 2t_1$.
The step-size for stage $1$ is $\alpha_1 = \lambda^{W_1}_N/2$ and for stage $k \in \{2,\ldots,K\}$ is $\alpha_k = \lambda^{W_1}_N/2^{2k+1}$. The momentum parameter for stage $k \in \{1,\ldots,K\}$ is $\beta_k = \frac{1-\sqrt{\alpha_k}}{1+\sqrt{\alpha_k}}$.

For $\texttt{SDA}$, we set the gossip matrix $L=I_N - W_{MH}$. The ``burn-in'' time $T_0$, the step-size $\eta$ and the momentum parameter $\zeta$ are chosen according to Theorem \ref{theorem:upper bound of stochastic dual accelerated method}. For \texttt{DSG}, we set the weight matrix $W = W_{MH}$. The step-size $\eta_t$ for $t \in \{0, 1, \ldots, T-1\}$ is chosen according to Theorem \ref{theorem:finite-time upper bound of Consensus-Based Stochastic Gradient Method}.

To simulate the distributed system, we set $N=100$ and $b = 10$. For each graph in Figure \ref{fig:5 networks}, we vary $K \in \{3, \ldots, 8\}$ to obtain different $T$. We repeat each algorithm $100$ times for every $T$ and graph, and report the average squared error as an approximation to the mean squared error.

\subsection{Details of experiments in Section \ref{sec:numerical results sample complexity}}
To simulate the distributed system, we set $N=100$ and $b = 1$. We consider the desired accuracy $\epsilon \in \{0.015, 0.0122, 0.01, 0.0082, 0.0067, 0.0055\}$ for the primary set of experiments (i.e. results in Figure \ref{fig:sample complexity large epsilon}), and $\epsilon \in \{0.00033, 0.00027, 0.00022, 0.00018, 0.00015, 0.00012\}$ for the complementary set of experiments (i.e. results in Figure \ref{fig:sample complexity smaller epsilon SDA}). Given an $\epsilon > 0$, we run each algorithm for $T \in \{1,2,\ldots\}$ iterations, and for each $T$, we repeat 100 times and calculate the average squared error. This process continues until the average squared error is less than or equal to $\epsilon$ for the first time. Then, we report the number of samples used by one node.

Both \texttt{SSTM\_sc} and \texttt{SDA} use the gossip matrix $L=I_N - W_{MH}$. For \texttt{SDA}, the parameters are chosen according to Theorem \ref{theorem:upper bound of stochastic dual accelerated method}. We set the parameters of \texttt{SSTM\_sc} according to Corollary 5.14 in \citet{gorbunov2019optimal} as follows. Let $L_{\psi} = \lambda_{1}(L)$, $\mu_{\psi} = \lambda_{N-1}(L)$, $\sigma^2_{\psi} = \lambda_1(L) N$ and $R_y = \norm{\mu - \bar{\mu}\cdot e_N}_2$, where $\mu = [\mu_1, \ldots, \mu_N]^\top$. We set the batch-size $r_t=\max\{1, \left(\mu_{\psi}/L_{\psi}\right)^{1.5} N^2 \sigma^2_{\psi} (1 + \sqrt{3\ln N})^2 /\epsilon\}$ for $t \in \{0, 1, \ldots, T-1\}$ and $r_T = \max\{1, \left(\mu_{\psi}/L_{\psi}\right)^{1.5} N \sigma^2_{\psi} (1 + \sqrt{3\ln N})^2 /\epsilon^2, \left(\mu_{\psi}/L_{\psi}\right)^{1.5} N^2 \sigma^2_{\psi} (1 + \sqrt{3\ln N})^2 R^2_y /\epsilon^2, \sigma^2_{\psi} (1 + \sqrt{3\ln N})^2 R^2_y /\epsilon^2\}$.

\subsection{Details of experiment in Section \ref{sec:numerical results non-asymptotic regime behavior}}
For this experiment, we set $b = 0$, which implies that $\mu_i = 0$ for all $i \in \mathcal{N}$. In this case, both \texttt{D-MASG} and \texttt{SDA} start from the optimal solution $\bar{\mu}=0$, and thus their deterministic errors remain zero all the time. Therefore, the mean squared errors of both methods are purely equal to their variances due to the noise of the samples. We consider star graph for this experiment, because the condition numbers of the gossip matrix $L = I_N - W_{MH}$ for \texttt{SDA} and the shifted matrix $W_1 = \frac{1}{2} I_N + \frac{1}{2} W_{MH}$ for \texttt{D-MASG} are approximately equal to $N$. So by choosing $N$ from $\{148, 190, 244, 314, 403, 518, 665\}$, we have a series of star graphs with ascending condition numbers. For each of those star graphs, we run \texttt{D-MASG} and \texttt{SDA} for $T = \sqrt{N}$ number of iterations, and repeat 100 times to report the average results. We set the parameters of $\texttt{SDA}$ according to Theorem \ref{theorem:upper bound of stochastic dual accelerated method}. For \texttt{D-MASG}, we set $K=1$, $t_1 = T$, $\alpha_1 = \lambda^{W_1}_N/2$ and $\beta_1 = \frac{1-\sqrt{\alpha_1}}{1+\sqrt{\alpha_1}}$.